\documentclass[10pt,11pt]{amsart}%
\usepackage[dvips]{graphicx}
\usepackage{amsmath}
\usepackage{color}
\usepackage{amsfonts}
\usepackage{amssymb}
\usepackage{version}%
\providecommand{\U}[1]{\protect\rule{.1in}{.1in}}
\providecommand{\U}[1]{\protect\rule{.1in}{.1in}}
\providecommand{\U}[1]{\protect\rule{.1in}{.1in}}
\providecommand{\U}[1]{\protect\rule{.1in}{.1in}}
\providecommand{\U}[1]{\protect\rule{.1in}{.1in}}
\providecommand{\U}[1]{\protect\rule{.1in}{.1in}}
\providecommand{\U}[1]{\protect\rule{.1in}{.1in}}
\newtheorem{theorem}{Theorem}[section]

\newtheorem{corollary}[theorem]{Corollary}

\newtheorem{example}[theorem]{Example}

\newtheorem{lemma}[theorem]{Lemma}

\newtheorem{proposition}[theorem]{Proposition}
\newtheorem{remark}[theorem]{Remark}
\let\oldremark\remark
\renewcommand{\remark}{\oldremark\normalfont}

\begin{document}
\title[Additive methods in associative algebras]{Additive combinatorics methods in associative algebras}
\date{June, 2015}
\author{Vincent Beck and C\'{e}dric Lecouvey}

\begin{abstract}
We adapt methods coming from additive combinatorics in groups to the study of
linear span in associative unital algebras. In particular, we establish for
these algebras analogues of Diderrich-Kneser's and Hamidoune's theorems on
sumsets and Tao's theorem on sets of small doubling. In passing we classify
the finite-dimensional algebras over infinite fields with finitely many
subalgebras. These algebras play a crucial role in our linear version of
Diderrich-Kneser's theorem. We also explain how the original theorems for
groups we linearize can be easily deduced from our results applied to group
algebras. Finally, we give lower bounds for the Minkowski product of two
subsets in finite monoids by using their associated monoid algebras.

\end{abstract}
\maketitle

\section{Introduction}

In this paper, we first establish analogues of theorems in additive
combinatorics on groups for a wide class of associative unital algebras. Next
we explain how this algebra setting permits to recover the original results on
groups and their analogues in fields but also yields similar lower bounds for
the Minkowski product of two subsets in monoids. Our results and tools mix
additive number theory, combinatorics, linear and commutative algebra and
basics considerations on Banach algebras.

Given $A$ and $B$ two non empty sets of a given group $G$, a classical problem
in additive combinatorics is to evaluate the cardinality $\left\vert
AB\right\vert $ of the Minkowski product $AB=\{ab\mid a\in A,b\in B\}$ in
terms of the cardinalities $\left\vert A\right\vert $ and $\left\vert
B\right\vert $.\ There exists a wide literature on this subject, notably a
famous result by Kneser (see \cite{Gr}, \cite{Nath}).

\begin{theorem}
[Kneser]Let $A$ and $B$ be finite subsets of the \emph{abelian} group $G$.
Then%
\[
\left\vert AB\right\vert \geq\left\vert A\right\vert +\left\vert B\right\vert
-\left\vert H\right\vert
\]
where $H=\{h\in G\mid hAB=AB\}$ is the stabilizer of $AB$ in $G.$
\end{theorem}

This theorem does not hold for non abelian groups and the question of finding
lower and upper bounds for product sets becomes then considerably more
difficult. Nevertheless, there exist in this case numerous weaker
results.\ Let us mention among them those of Diderrich \cite{Die}, Olson
\cite{Ols} and Tao \cite{Tao}, \cite{Tao3} we shall evoke in more details in
Section \ref{Sec_Kneser}.

Analogous estimates exist in the context of fields and division rings. As far
as we are aware, this kind of generalizations was considered for the first
time in \cite{Xian} and \cite{Kain}. Consider $K$ a field extension of the
field $\mathrm{k}$ and $A$ a finite subset in $K$.\ Write $\mathrm{k}\langle
A\rangle$ for the $\mathrm{k}$-subspace of $K$ generated by $A$ and let
$\dim_{\mathrm{k}}(A)$ be its dimension. For $A,B$ two finite subsets of $K$,
we set $AB=\{ab\mid a\in A,b\in A\}$. Then $\dim_{\mathrm{k}}(AB)$, the
dimension of $\mathrm{k}\langle AB\rangle$, is finite. The following analogue
of Kneser's theorem for fields is proved in \cite{Xian} and \cite{Kain}.

\begin{theorem}
Let $K$ be a commutative extension of $\mathrm{k}$. Assume every algebraic
element in $K$ is separable over $\mathrm{k}$.\ Let $A$ and $B$ be two
nonempty finite subsets of $K^{\ast}$. Then
\begin{equation}
\dim_{\mathrm{k}}(AB)\geq\dim_{\mathrm{k}}(A)+\dim_{\mathrm{k}}(B)-\dim
_{\mathrm{k}}(H) \label{ineqLK}%
\end{equation}
where $H:=\{h\in K\mid h\mathrm{k}\langle AB\rangle\subseteq\mathrm{k}\langle
AB\rangle\}$.
\end{theorem}

\noindent Here $H$ is an intermediate field containing $\mathrm{k}$ and the
separability hypothesis is crucial since the proof uses the fact that $K$
admits only a finite number of finite extensions of $\mathrm{k}$ (which is
also assumed in \cite{Kain}).\ Equivalently, this theorem asserts that the sum
of the dimensions of $\mathrm{k}\langle AB\rangle$ and its stabilizer must be
at least equal to the sum of dimensions of $\mathrm{k}\langle A\rangle$ and
$\mathrm{k}\langle B\rangle$. Remarkably, the authors showed that their
theorem implies Kneser's theorem for abelian groups by using Galois
correspondence. Observe it is not known if the theorem remains valid without
the separability hypothesis (see \cite{Hou}). Non commutative analogues of
this theorem were established in \cite{EL2} (linear version of Olson's theorem
without any separability hypothesis) and in \cite{Lec} (linear version of
Diderrich's theorem where only the elements of the set $A$ are assumed
pairwise commutative).

In \cite{Lec}, linear analogues (i.e. in division rings and fields) of
theorems by Pl\"{u}nnecke and Ruzsa \cite{Ruz} are given yielding upper bounds
for $\dim_{\mathrm{k}}(AB)$. In passing we observe these theorems can be
adapted to some unital associative algebras.\ It is then a natural question to
ask wether lower bound estimates for $\dim_{\mathrm{k}}(AB)$ similar to
(\ref{ineqLK}) exists for subspaces of a unital associative $\mathrm{k}%
$-algebra $\mathcal{A}$. A first obstruction is due to the existence of non
trivial annihilators of subsets.\ Indeed if the right annihilator
$\mathrm{ann}_{r}(A)$ of $A$ is not reduced to $\{0\}$, we can take for $B$
any generating subset of $\mathrm{ann}_{r}(A)$ and obtain $\dim_{\mathrm{k}%
}(AB)=0$.\ To overcome this problem, we will assume most of the time that the
$\mathrm{k}$-subspaces we consider in $\mathcal{A}$ contains at least one
invertible element.\ We thus have $\mathrm{ann}_{r}(\mathrm{k}\langle
A\rangle)=\mathrm{ann}_{l}(\mathrm{k}\langle A\rangle)=\{0\}$. We prove in
this paper that, quite surprisingly, this suffices to establish in
$\mathcal{A}$ an analogue of Diderrich's theorem but also analogues of
estimates by Hamidoune and Tao. To obtain lower bounds similar to
(\ref{ineqLK}), it nevertheless remains a second obstruction.\ We indeed need
an analogue of the separability hypothesis in our algebras context. In fact we
shall see that it suffices to assume that the subalgebra of $\mathcal{A}$
generated by $A$ has finitely many finite-dimensional subalgebras.\ This leads
us to classify the f.d. associative unital algebras with finitely many
subalgebras in Section \ref{Sec_Subalgebras}.

\bigskip

The Paper is organized as follows. In Section 2 we precise the algebra setting
we consider. Also to get a sufficient control on the invertible elements of
the algebra $\mathcal{A}$, we need to assume in the theorems we establish that
$\mathrm{k}$ is infinite and $\mathcal{A}$ satisfies one of the two following
(strong or weak) hypotheses:

\begin{description}
\item[H$_{\mathrm{s}}$] $\mathcal{A}$ is finite-dimensional over $\mathrm{k}$
or a Banach algebra or a finite product of (possibly infinite-dimensional)
field extensions over $\mathrm{k}$.

\item[H$_{\mathrm{w}}$] $\mathcal{A}$ is finite-dimensional over $\mathrm{k}$
\ or a subalgebra of a Banach algebra or a subalgebra of a finite product of
(possibly infinite-dimensional) field extensions over $\mathrm{k}$. Equivalently, an algebra verifying \textbf{H}$_{\mathrm{w}}$ is a subalgebra of an algebra verifying \textbf{H}$_{\mathrm{s}}$.
\end{description}

These hypotheses are not optimal and one can establish some refinements of our
results we will not detail for simplicity.\ The main result of Section 3 is
the classification of f.d.\negthinspace algebras with finitely many
subalgebras.\ Section 4 is devoted to the analogue of Diderrich's theorem. We
notably obtain a lower bound similar to (\ref{ineqLK}) where $\mathcal{H}$ is
the subalgebra of $\mathcal{A}$ which stabilizes $\mathrm{k}\langle AB\rangle
$.\ In Section 5, we establish analogues of results by Tao on spaces of small
doubling using a linear version of Hamidoune connectivity. Finally, in Section
6, we explain how the original theorems of Kneser and Diderrich in a group $G$
can be very easily recovered from our linear version in the group algebra of
$G$. In particular, the link with the group setting does not require to
realize $G$ as the Galois group of a finite extension of $\mathrm{k}$ as in
\cite{Xian} which would become problematic when $G$ is non abelian.\ We also
explain, in Section~7, how it is possible to state Hamidoune type results in
finite monoids by considering their monoid algebras.

\bigskip

While writing this paper, we were informed by G. Z\'{e}mor that a Kneser type
theorem has been very recently obtained in \cite{MZ} for the algebra
$\mathcal{A}=\mathrm{k}^{n}$ with applications to linear code theory.

\bigskip

\noindent\textbf{Acknowledgments.} Both authors are partly supported by the
"Agence Nationale de la Re\-cherche" grant ANR-12-JS01-0003-01 ACORT.

\section{The algebra setting}

\subsection{Vector span in an algebra}

Let $\mathcal{A}$ be a unital associative algebra over the field $\mathrm{k}$.
We denote by $\mathcal{A}_{\ast}=\mathcal{A}\setminus\{0\}$ and by
$U(\mathcal{A)}$ the group of invertible elements in $\mathcal{A}$. All along
this paper, by a subalgebra $\mathcal{B}$ of $\mathcal{A}$, we always mean a
unital subalgebra which contains $1$.

For any subset $A$ of $\mathcal{A}$, let $\mathrm{k}\langle A\rangle$ be the
$\mathrm{k}$-subspace of $\mathcal{A}$ generated by $A$. We write $\left\vert
A\right\vert $ for the cardinality of $A$, and $\dim_{\mathrm{k}}(A)$ for the
dimension of $\mathrm{k}\langle A\rangle$ over $\mathrm{k}$.\ When $\left\vert
A\right\vert $ is finite, $\dim_{\mathrm{k}}(A)$ is also finite and we have
$\dim_{\mathrm{k}}(A)\leq\left\vert A\right\vert $. We denote by
$\mathbb{A}(A)\subseteq\mathcal{A}$ the subalgebra generated by $A$ in
$\mathcal{A}$.

Given subsets $A$ and $B$ of $\mathcal{A}$, we thus have $\mathrm{k}\langle
A\cup B\rangle=\mathrm{k}\langle A\rangle+\mathrm{k}\langle B\rangle$, the sum
of the two spaces $\mathrm{k}\langle A\rangle$ and $\mathrm{k}\langle
B\rangle$. We have also $\mathrm{k}\langle A\cap B\rangle\subseteq
\mathrm{k}\langle A\rangle\cap\mathrm{k}\langle B\rangle$ and $\mathrm{k}%
\langle AB\rangle=\mathrm{k}\langle\mathrm{k}\langle A\rangle\mathrm{k}\langle
B\rangle\rangle$.\ We write as usual
\[
AB:=\{ab\mid a\in A,b\in B\}
\]
for the Minkowski product of the sets $A$ and $B$.\ Given a family of nonempty
subsets $A_{1},\ldots,A_{n}$ of $\mathcal{A}$, we define the Minkowski product
$A_{1}\cdots A_{n}$ similarly.\ 

Any finite-dimensional $\mathrm{k}$-subspace $V$ of $\mathcal{A}$ can be
realized as $V=\mathrm{k}\langle A\rangle$, where $A$ is any finite subset of
nonzero vectors spanning $V$. Also, when $V_{1}$ and $V_{2}$ are two
$\mathrm{k}$-vector spaces in $K$, $V_{1}V_{2}\subseteq\mathrm{k}\langle
V_{1}V_{2}\rangle$ but $V_{1}V_{2}$ is not a vector space in general. We set
$U(V):=V\cap U(\mathcal{A)}$ and $U(V)^{-1}=\{x^{-1}\mid x\in U(V)\}$. In what
follows we denote by $A,B$ subsets of $\mathcal{A}$ whereas $V,W$ refer to
$\mathrm{k}$-subspaces of $\mathcal{A}$.

We aim to give some estimates of $\dim_{\mathrm{k}}(AB)$ in terms of
$\dim_{\mathrm{k}}(A),$ $\dim_{\mathrm{k}}(B)$ and structure constants
depending on the algebra $\mathcal{A}$ (typically the dimensions of some
finite-dimensional subalgebras of $\mathcal{A}$). More generally we consider
similar problems for $\dim_{\mathrm{k}}(A_{1}\cdots A_{r})$ where
$A_{1},\ldots,A_{r}$ are finite subsets of $\mathcal{A}$. The following is
straightforward%
\[
\max(\dim_{\mathrm{k}}(A),\dim_{\mathrm{k}}(B))\leq\dim_{\mathrm{k}}%
(AB)\leq\dim_{\mathrm{k}}(A)\dim_{\mathrm{k}}(B)
\]
when $\mathrm{k}\langle A\rangle$ and $\mathrm{k}\langle B\rangle$ contain at
least an invertible element.\ In the sequel, we will restrict ourselves for
simplicity to the case where $\mathrm{k}$ is infinite. For $\mathrm{k}$ a
finite field, we can obtain estimates for $\dim_{\mathrm{k}}(AB)$ from the
infinite field case by considering the algebra $\mathcal{A}^{\prime
}=\mathcal{A}\otimes_{\mathrm{k}}\mathrm{k}(t)$ where $\mathrm{k}(t)$ is field
of rational functions in $t$ over $\mathrm{k}$. We indeed then have
\[
\dim_{\mathrm{k}}(AB)=\dim_{\mathrm{k}(t)}(A^{\prime}B^{\prime}),\quad
\dim_{\mathrm{k}}(A)=\dim_{\mathrm{k}(t)}(A^{\prime})\quad\text{ and }%
\quad\dim_{\mathrm{k}}(B)=\dim_{\mathrm{k}(t)}(B^{\prime})
\]
where $A^{\prime}=A\otimes_{\mathrm{k}}1\in\mathcal{A}\otimes\mathrm{k}(t)$
and $B^{\prime}=B\otimes_{\mathrm{k}}1\in\mathcal{A}\otimes\mathrm{k}%
(t)$.\ The following elementary lemma will be useful.

\begin{lemma}
\label{lem_stab}Let $\mathcal{A}$ be a finite-dimensional algebra over the
field $\mathrm{k}$ and $A$ be a finite subset of $\mathcal{A}$ such that
$A\cap U(\mathcal{A})\neq\emptyset$ and $\mathrm{k}\langle A^{2}%
\rangle=\mathrm{k}\langle A\rangle$.\ Then $\mathrm{k}\langle A\rangle$ is a
subalgebra $\mathcal{A}$ of and $U(\mathrm{k}\langle A\rangle)=U(\mathcal{A}%
)\cap\mathrm{k}\langle A\rangle$.\ 
\end{lemma}

\begin{proof}
Observe that $\mathrm{k}\langle A^{2}\rangle=\mathrm{k}\langle A\rangle$ means
that $\mathrm{k}\langle A\rangle$ is closed under multiplication. Then, for
any nonzero $a\in\mathrm{k}\langle A\rangle$, the map $\varphi_{a}%
:\mathrm{k}\langle A\rangle\rightarrow\mathrm{k}\langle A\rangle$ which sends
$\alpha\in\mathrm{k}\langle A\rangle$ on $\varphi_{a}(\alpha)=a\alpha$ is a
$\mathrm{k}$-linear endomorphism of the space $\mathrm{k}\langle A\rangle
$.\ If we choose $a\in\mathrm{k}\langle A\rangle\cap U(\mathcal{A})$,
$\varphi_{a}$ is a $\mathrm{k}$-linear injective endomorphism of the
finite-dimensional space $\mathrm{k}\langle A\rangle$.\ Hence it is an
automorphism.\ There then exists $\alpha\in\mathrm{k}\langle A\rangle$ such
that $a\alpha=a$.\ Since $a\in U(\mathcal{A})$, this shows that $\alpha
=1\in\mathrm{k}\langle A\rangle$ and $\mathrm{k}\langle A\rangle$ is a unital
subalgebra of $\mathcal{A}$.\ Now since $1\in\mathrm{k}\langle A\rangle$,
there exists $\beta\in\mathrm{k}\langle A\rangle$ such that $a\beta=1$. So
$a^{-1}=\beta\in\mathrm{k}\langle A\rangle$ and $U(\mathrm{k}\langle
A\rangle)=U(\mathcal{A})\cap\mathrm{k}\langle A\rangle$.\ 
\end{proof}

\bigskip

For any subset $A$ in $\mathcal{A}_{\ast}$, we set%
\[
\mathcal{H}_{l}(A):=\{h\in\mathcal{A}\mid h\,\mathrm{k}\langle A\rangle
\subseteq\mathrm{k}\langle A\rangle\}\quad\text{ and }\quad\mathcal{H}%
_{r}(A):=\{h\in\mathcal{A}\mid\mathrm{k}\langle A\rangle\,h\subseteq
\mathrm{k}\langle A\rangle\}
\]
for the left and right stabilizers of $\mathrm{k}\langle A\rangle$ in
$\mathcal{A}$. Clearly $\mathcal{H}_{l}(A)$ and $\mathcal{H}_{r}(A)$ are
subalgebras.\ In particular, when $\mathcal{A}$ is commutative, $\mathcal{H}%
_{l}(A)=\mathcal{H}_{r}(A)$ is a commutative $\mathrm{k}$-algebra that we
simply write $\mathcal{H}(A)$.\ If $\mathcal{H}_{l}(A)$ (resp. $\mathcal{H}%
_{r}(A)$) is not equal to $\mathrm{k}$, we say that $\mathrm{k}\langle
A\rangle$ is \emph{left periodic} (resp. \emph{right periodic}). Observe also
that for $A$ and $B$ two finite subsets of $\mathcal{A}$, if $\mathrm{k}%
\langle A\rangle$ is left periodic (resp. $\mathrm{k}\langle B\rangle$ is
right periodic), then $\mathrm{k}\langle AB\rangle$ is left periodic (resp.
right periodic). Indeed, for $\mathrm{k}\langle A\rangle$ left periodic, we
have $\mathcal{H}_{l}(A)\neq\mathrm{k}$ and $\mathcal{H}_{l}(A)\mathrm{k}%
\langle A\rangle\subseteq\mathrm{k}\langle A\rangle$.\ By linearity of the
multiplication in $\mathcal{A}$, this gives $\mathcal{H}_{l}(A)\mathrm{k}%
\langle AB\rangle\subseteq\mathrm{k}\langle AB\rangle$ thus $\mathcal{H}%
_{l}(A)\subseteq\mathcal{H}_{l}(AB)$ and $\mathcal{H}_{l}(AB)\neq\mathrm{k}%
$.\ The case $\mathrm{k}\langle B\rangle$ right periodic is similar.

\begin{remark}
The stabilizer algebra $\mathcal{H}_{l}(A)$ (resp. $\mathcal{H}_{r}(A)$) may
also be described as the biggest subalgebra of $\mathcal{A}$ such that
$\mathrm{k}\langle A\rangle$ is a left (resp. right) representation for this subalgebra.

Assume $A$ is a finite subset of $\mathcal{A}_{\ast}$ such that $A\cap
U(\mathcal{A})\neq0$. Then $\mathcal{H}_{l}(A)$ and $\mathcal{H}_{r}(A)$ are
finite-dimensional $k$-subalgebras of $\mathcal{A}$.\ Indeed, for any $a\in
A\cap U(\mathcal{A})$, we have $\mathcal{H}_{l}(A)a\subset\mathrm{k}\langle
A\rangle$ and $a\mathcal{H}_{r}(A)\subset\mathrm{k}\langle A\rangle$ with
\[
\dim_{\mathrm{k}}(\mathcal{H}_{l}(A)a)=\dim_{\mathrm{k}}(\mathcal{H}%
_{l}(A))\quad\text{and} \quad\dim_{\mathrm{k}}(a\mathcal{H}_{r}(A))=\dim
_{\mathrm{k}}(\mathcal{H}_{r}(A)) .
\]

\end{remark}

For any subset $A$ in $\mathcal{A}_{\ast}$, we define%
\[
\mathrm{ann}_{l}(A):=\{a\in\mathcal{A}\mid a\,\mathrm{k}\langle A\rangle
=\{0\}\}\quad\text{and} \quad\mathrm{ann}_{r}(A):=\{a\in\mathcal{A}%
\mid\mathrm{k}\langle A\rangle\,a=\{0\}\}
\]
for the left and right annihilator of $\mathrm{k}\langle A\rangle$ in
$\mathcal{A}$. Observe $\mathrm{ann}_{l}(A)$ and $\mathrm{ann}_{r}(A)$ are not
subalgebras of $\mathcal{A}$ since they do not contain $1$ but respectively a
left ideal and a right ideal of $\mathcal{A}$. Moreover $\mathrm{ann}_{l}(A)$
(resp. $\mathrm{ann}_{r}(A)$) is a two-sided ideal of $\mathcal{H}_{l}(A)$
(resp. $\mathcal{H}_{r}(A)$) and $\mathrm{k}\oplus\mathrm{ann}_{l}(A)$ and
$\mathrm{k}\oplus\mathrm{ann}_{r}(A)$ are respectively subalgebras of
$\mathcal{H}_{l}(A)$ and $\mathcal{H}_{r}(A)$.

When $\mathcal{A}$ is commutative, we write $\mathrm{ann}_{l}(A)=\mathrm{ann}%
_{r}(A)=\mathrm{ann}(A)$. Also when $A=\{x_{1},\ldots,x_{r}\}$ is finite, we
have
\[
\mathrm{ann}_{l}(A)=\bigcap\limits_{i=1}^{r}\mathrm{ann}_{l}(x_{i}).
\]


\subsection{Basis of invertible elements}

Let $\mathcal{A}$ be an algebra over the field $\mathrm{k}$. The algebra
$\mathcal{A}$ has no non trivial finite-dimensional subalgebra when for any
$a$ in $\mathcal{A} \setminus k,$ the algebra morphism%
\[
\theta_{a}:\left\{
\begin{array}
[c]{c}%
\mathrm{k}[T]\rightarrow\mathcal{A}\\
P\longmapsto P(a)
\end{array}
\right.
\]
is injective.\ This means that $\mathrm{k}[a]=\operatorname{Im}\theta_{a}$ is
isomorphic to $\mathrm{k}[T]$.\ 

When $\mathcal{A}$ is finite-dimensional, for any element $a\in\mathcal{A}$,
the $\mathrm{k}$-subalgebra $\mathrm{k}[a]$ generated by $a$ is isomorphic to
$\mathrm{k}[T]/(\mu_{a})$ where $\mu_{a}$ is the minimal polynomial of $a$ and
$\ker\theta_{a}=(\mu_{a})$. In particular, $\mathrm{k}[a]$ is a field if and
only if $\mu_{a}$ is irreducible over $\mathrm{k}$.\ 

\begin{lemma}
\label{Lem_finite}Assume $\mathcal{A}$ is finite-dimensional and consider
$a\in A$ such that $a\notin U(\mathcal{A})$. Then, there exists $P\in
\mathrm{k}[T]$ such that $aP(a)=0$ and $P(a) \neq0$.
\end{lemma}

\begin{proof}
Since $a\notin U(\mathcal{A})$, the minimal polynomial $\mu_{a}$ is divisible
by $T$. Let us write $\mu_{a}=TP(T)$ with $P(T) \in k[T]$. We have $P(a)
\neq0$ since $\deg P < \deg\mu_{a}$ and $aP(a)=\mu_{a}(a)=0$.
\end{proof}

\bigskip

Recall the algebras we shall consider are unital and associative over the
infinite field $\mathrm{k}$. In addition we will restrict ourself most of the
time to finite-dimensional algebras, Banach algebras over $\mathrm{k}%
=\mathbb{R}$ or $\mathrm{k}=\mathbb{C}$ or finite product of field extensions
over $\mathrm{k}$.\ In the case of Banach algebras, we will write $\left\Vert
\cdot\right\Vert $ for the ambient norm.


\begin{lemma}
\label{Lemma_Banach}Assume $\mathcal{A}$ is a Banach algebra and consider $a$
in $\mathcal{A}$ such that $\left\Vert a\right\Vert <1$. Then $1-a\in
U(\mathcal{A})$.
\end{lemma}

\begin{proof}
Since $\left\Vert a\right\Vert <1$ and $\mathcal{A}$ is a complete space, we
have
\[
(1-a)^{-1}=\sum_{k=0}^{+\infty}a^{k}%
\]
\end{proof}


\bigskip

\begin{lemma}
Assume the algebra $\mathcal{A}$ satisfies \textbf{H}$_{\mathrm{s}}$ and
consider $a\in\mathcal{A}$. There exist infinitely many $\lambda\in\mathrm{k}$
\ such that the elements of the form $a-\lambda1$ belong to $U(\mathcal{A})$.
\end{lemma}

\begin{proof}
Assume first that $\mathcal{A}$ is finite-dimensional.\ Let $\lambda$ such
that $a-\lambda1$ is not invertible. By Lemma~\ref{Lem_finite}, there exists
$P\in\mathrm{k}[T]$ such that $(a-\lambda1)P(a-\lambda1)=0$ and $P(a-\lambda1)
\neq0$.\ Set $Q(T)=P(T-\lambda)$. We then get, $(a-\lambda1)Q(a)=0$ and
$Q(a)\neq0$.\ This can be rewritten $aQ(a)=\lambda Q(a)$ with $Q(a)\neq
0$.\ Therefore $Q(a)$ is an eigenvector associated to the eigenvalue $\lambda$
for the linear map $\varphi_{a}:\mathcal{A\rightarrow A}$ defined by
$\varphi_{a}(x)=ax$ for any $x\in\mathcal{A}$.\ Since $\mathcal{A}$ is
finite-dimensional, the linear map $\varphi$ can only admit a finite number of
eigenvalues and we are done.


Now assume $\mathcal{A}$ is a Banach algebra.\ For any $\left\vert
\lambda\right\vert >\left\Vert a\right\Vert $, Lemma \ref{Lemma_Banach} shows
that $1-\lambda^{-1}a\in U(A)$. Hence $-\lambda(1-\lambda^{-1}a)=a-\lambda1\in
U(A)$.

Let us now consider the third case: assume that $\mathcal{A}=K_{1}
\times\cdots\times K_{m}$ is a product of field extensions over $\mathrm{k}$.
Let $a=(a_{1},\ldots, a_{m}) \in\mathcal{A}$ and choose $\lambda\in\mathrm{k}$
distinct from $a_{1},\ldots, a_{m}$.
\end{proof}

\begin{proposition}
\label{Prop_base_inv} Assume $\mathcal{A}$ satisfies \textbf{H}$_{\mathrm{s}}$
and let $A$ be a finite subset in $\mathcal{A}$ such that $A\cap
U(\mathcal{A})\neq\emptyset$. Then the $\mathrm{k}$-subspace $\mathrm{k}%
\langle A\rangle$ admits a basis of invertible elements.
\end{proposition}

\begin{proof}
Let $a\in A\cap U(\mathcal{A})$. By replacing $A$ by $a^{-1}A$, we can assume
that $1\in A$.\ The $\mathrm{k}$-subspace $\mathrm{k}\langle A\rangle$ admits
a basis containing $1$ of the form $B=\{1,b_{2},\ldots b_{d}\}$ with
$d=\dim_{\mathrm{k}}A$.\ By using the previous lemma, there exits $\lambda_{i}
\in\mathrm{k}$ such that each $b_{i}-\lambda_{i}1$ is invertible.\ Then
$B^{\prime}=\{1,b_{2}-\lambda_{1} 1,\ldots,b_{d}-\lambda_{d} 1\}$ is a
$\mathrm{k}$-basis of $\mathrm{k}\langle A\rangle$ containing only invertible
elements.
\end{proof}


\begin{lemma}
\label{Lemma_VDM}Assume $\mathcal{A}$ satisfies \textbf{H}$_{\mathrm{s}}$. Let
$V$ be a $n$-dimensional subspace of $\mathcal{A}$ such that $V\cap
U(\mathcal{A})\neq\emptyset$.\ Consider $\{x_{1},\ldots,x_{n}\}$ a basis of
$V$ over $\mathrm{k}$ with $x_{1}$ invertible. Then

\begin{enumerate}
\item Any $n$ vectors in the set
\[
X=\{x_{1}+\alpha x_{2}+\cdots+\alpha^{n-1}x_{n}\mid\alpha\in\mathrm{k}\}
\]
form a basis of $V$ over $\mathrm{k}$.

\item The set $X$ contains an infinite number of invertible elements, that is
an infinite number of elements of the form $x_{1}+\alpha x_{2}+\cdots
+\alpha^{n-1}x_{n},\alpha\in\mathrm{k}$ are invertible.

\item The set $X$ contains a basis of $V$ over $\mathrm{k}$ of invertible elements.
\end{enumerate}
\end{lemma}

\begin{proof}
Assertion 1 is an application of the Vandermonde determinant.

For assertion 2, assume first $\mathcal{A}$ is finite-dimensional.\ For any
$\alpha\in\mathrm{k}$, set $x_{\alpha}=x_{1}+\alpha x_{2}+\cdots+\alpha
^{n-1}x_{n}$.\ Let $\varphi_{\alpha}:\mathcal{A}\rightarrow\mathcal{A}$ be the
left multiplication by $x_{\alpha}$ in $\mathcal{A}$.\ Clearly $x_{\alpha}$ is
invertible if and only if the linear map $\varphi_{\alpha}$ is an isomorphism
of $\mathrm{k}$-spaces. Write $P(\alpha)=\det\varphi_{\alpha}$ for the
determinant of the linear map $\varphi_{\alpha}$. Then $P(\alpha)$ is a non
zero polynomial in $\alpha$ since $P(0)\neq0$. So $P(\alpha)=0$ only for a
finite number of $\alpha\in\mathrm{k}$ and we are done.

Now assume that $\mathcal{A}$ is a Banach algebra. Observe that $x_{a}$ is
invertible if and only if $y_{\alpha}=x_{1}^{-1}x_{\alpha}$ is.\ We have
$y_{\alpha}=1+\alpha x_{1}^{-1}x_{2}+\cdots+\alpha^{n-1}x_{1}^{-1}x_{n}$.
Assume $\left\vert \alpha\right\vert \leq1$. We get
\[
\left\Vert 1-y_{\alpha}\right\Vert =\left\Vert \alpha x_{1}^{-1}x_{2}%
+\cdots+\alpha^{n-1}x_{1}^{-1}x_{n}\right\Vert \leq\left\vert \alpha
\right\vert (\left\Vert x_{1}^{-1}x_{2}\right\Vert +\cdots+\left\Vert
x_{1}^{-1}x_{n}\right\Vert ).
\]
Therefore $\left\Vert 1-y_{\alpha}\right\Vert <1$ for any $\alpha$ such that
$\left\vert \alpha\right\vert <(\left\Vert x_{1}^{-1}x_{2}\right\Vert
+\cdots+\left\Vert x_{1}^{-1}x_{n}\right\Vert )^{-1}$. By using Lemma
\ref{Lemma_Banach} we obtain our assertion 2 for Banach algebras since there
is infinitely many such $\alpha$ in $\mathbb{R}$ and $\mathbb{C}$.

Let us consider the third case: $\mathcal{A}=K_{1}\times\cdots\times K_{m}$ is
a product of field extensions over $\mathrm{k}$. We write $x_{i}%
=(a_{i1},\ldots,a_{im})$ for $i\in\{1,\ldots,n\}$. Then $x_{1}+\alpha
x_{2}+\cdots+\alpha^{n-1}x_{n}$ is invertible if and only if for every
$j\in\{1,\ldots,m\},\ P_{j}(\alpha):=a_{1j}+\alpha a_{2j}+\cdots+\alpha
^{n-1}a_{nj}\neq0$. But $P_{j}\in K_{j}[X]$ is a non zero polynomial (since
$a_{1j}\neq0$) and hence has only a finite number of roots in the
(commutative) field $K_{j}$ and thus also in $\mathrm{k}$.

Assertion 3 is a consequence of Assertions 1 and 2.
\end{proof}

\section{Algebras with finitely many subalgebras}

\label{Sec_Subalgebras}We resume the notation and the hypotheses of the
previous section on the algebra $\mathcal{A}$. The goal of this section is to
classify the finite-dimensional algebras with finitely many subalgebras. Such
algebras will indeed appear in the Kneser type theorem we shall state in
Section \ref{Sec_Kneser}.

\subsection{Primitive element}

The first step of our classification is to show that a finite-dimensional
algebras with finitely many subalgebras is generated by one element.



\begin{lemma}
[Union of subspaces]\label{lem-union-sbspace} Let $V$ be a finite-dimensional
vector space over $\mathrm{k}$ and $V_{1},\ldots,V_{n}$ proper subspaces of
$V$. Then
\[
\bigcup_{i=1}^{n}V_{i}\varsubsetneq V
\]

\end{lemma}

\begin{proof}
Since $V_{i}$ is a proper subspace of $V$, it can be embedded in a hyperplane
$H_{i}$ of $V$ which is the kernel of a linear form $\varphi_{i}\in V^{\ast
}\setminus\{0\}$. Since the ring of polynomial functions on $V$ is an integral
domain ($\mathrm{k}$ is infinite), then
\[
f=\prod_{i=1}^{n}\varphi_{i}%
\]
is a non zero function. Any vector $v\in V$ such that $f(v)\neq0$ is not in
the union of the $V_{i}$.
\end{proof}

\begin{corollary}
\label{cor-monogene} Let $\mathcal{A}$ be a finite-dimensional algebra over
$\mathrm{k}$ such that $\mathcal{A}$ has only a finite number of subalgebras.
Then there exists $x\in\mathcal{A}$ such that $\mathcal{A}=\mathrm{k}[x]$. In
particular, $\mathcal{A}$ is commutative and generated by only one element.
\end{corollary}

\begin{proof}
Lemma~\ref{lem-union-sbspace} ensures us that there exists $x\in\mathcal{A}$
which is not in any proper subalgebra of $\mathcal{A}$. We then get
$\mathcal{A}=\mathrm{k}[x]$.
\end{proof}

\subsection{Structure of finite-dimensional algebras with finitely many
subalgebras}

The rest of this section is devoted to the study of algebras with only a
finite number of subalgebras. Our aim is to prove a classification theorem for
this kind of algebras (Theorem~\ref{TH_classification}). The proof is divided
in two steps. In the first step, we reduce through various easy lemmas the
form for algebras with finite number of subalgebras. The second step shows
that algebras of the form obtained in the first step has indeed a finite
number of subalgebras.

Let us start our first step. Corollary~\ref{cor-monogene} says that we can
restrict our attention to algebras of the form $\mathrm{k}[T]/(P)$. We begin
with an easy remark which will be very useful.

\begin{remark}
\textbf{(Quotient -- Subalgebra)}\label{rem-quotient} If $\mathcal{A}$ is a
finite-dimensional algebra over $\mathrm{k}$ such that $\mathcal{A}$ has only
a finite number of subalgebras. Then every subalgebra or quotient
$\mathcal{B}$ of $\mathcal{A}$ verifies the same property. This is obvious for
subalgebras. For the quotient case, the subalgebras of $\mathcal{B}$ are in
bijection with the subalgebras of $\mathcal{A}$ containing the kernel of the
surjective map from $\mathcal{A}$ to~$\mathcal{B}$.
\end{remark}

Let us now construct finite-dimensional algebras with an infinite number of subalgebras.

\begin{lemma}
\label{lem-kxxn} For $n\geq4$, $\mathrm{k}[T]/T^{n}$ has an infinite number of subalgebras.
\end{lemma}

\begin{proof}
Due to the remark~\ref{rem-quotient}, it suffices to study the case $n=4$. In
this case, for $\lambda\neq\mu\in\mathrm{k}$, the subalgebras $\mathcal{A}%
_{\lambda}=\mathrm{k}[T^{2}+\lambda T^{3}]$ and $\mathcal{A}_{\mu}%
=\mathrm{k}[T^{2}+\mu T^{3}]$ are distinct subalgebras. Indeed, they are two
dimensional algebras isomorphic to $\mathrm{k}[T]/T^{2}$. So $(1,T^{2}+\lambda
T^{3})$ is a basis for $\mathcal{A}_{\lambda}$ and $T^{2}+\mu T^{3}$ can not
be written as a linear combination of $1$ and $T^{2}+\lambda T^{3}$ in
$\mathrm{k}[T]/T^{4}$. For if evaluating at $T=0$, $T^{2}+\lambda T^{3}$ and
$T^{2}+\mu T^{3}$ would be colinear.
\end{proof}

\begin{lemma}
For $n\geq4$ and $P\in\mathrm{k}[T]$ a non constant polynomial, $\mathrm{k}%
[T]/P^{n}$ has an infinite number of subalgebras.
\end{lemma}

\begin{proof}
Indeed, the subalgebra generated by $P$ is isomorphic to $\mathrm{k}[T]/T^{n}%
$. So~Remark~\ref{rem-quotient} and Lemma~\ref{lem-kxxn} give the result.
\end{proof}

\begin{lemma}
\label{lem-deg2} For $n=2,3$ and $P\in\mathrm{k}[T]$ with $\deg P\geq2$ then
$\mathrm{k}[T]/P^{n}$ has an infinite number of subalgebras.
\end{lemma}

\begin{proof}
Thanks to remark~\ref{rem-quotient}, it suffices to consider the case $n=2$.
For $\lambda\in\mathrm{k}$, let us consider $Q_{\lambda}=(1+\lambda
T)P\in\mathrm{k}[T]/P^{2}$. The subalgebra $\mathcal{A}_{\lambda}$ of
$\mathrm{k}[T]/P^{2}$ generated by $Q_{\lambda}$ is isomorphic to
$\mathrm{k}[T]/T^{2}$ and so is two dimensional. For $\lambda\neq\mu$, we have
$\mathcal{A}_{\lambda}\neq\mathcal{A}_{\mu}$. Indeed, if $(1+\mu
T)P=\alpha+\beta(1+\lambda T)P$ in $\mathrm{k}[T]/P^{2}$. Then, going into
$\mathrm{k}[T]/P$, we get $\alpha=0$ and so $P\mid(1+\mu T)-\beta(1+\lambda
T)$. Since $\deg P\geq2$, we obtain $(1+\mu T)=\beta(1+\lambda T)$, which is absurd.
\end{proof}

\begin{lemma}
\label{lem-nilpotent} For $n,m\in\{2,3\}$, $\mathrm{k}[T]/T^{n}\times
\mathrm{k}[T]/T^{m}$ has an infinite number of subalgebras.
\end{lemma}

\begin{proof}
Thanks to remark~\ref{rem-quotient}, it suffices to consider the case $m=n=2$.
For $\lambda\in\mathrm{k}$, the element $(T,\lambda T)$ generates an algebra
isomorphic to $\mathrm{k}[T]/T^{2}$ denoted by $\mathcal{A}_{\lambda}$. For
$\lambda\neq\mu$, we have $\mathcal{A}_{\lambda}\neq\mathcal{A}_{\mu}$.
Indeed, if $(T,\mu T)=\alpha(1,1)+\beta(T,\lambda T)$. Then, mapping $T$ to
$0$, we get $\alpha=0$ and $(T,\mu T)$ and $(T,\lambda T)$ are colinear which
is not the case.
\end{proof}

\begin{corollary}
\label{cor-structure} Let $\mathcal{A}$ be a finite-dimensional algebra over
$\mathrm{k}$ such that $\mathcal{A}$ has only a finite number of subalgebras.
Then, there exist finite algebraic extensions of $\mathrm{k}$, $L_{1}%
,\ldots,L_{n}$ generated over $\mathrm{k}$ by one element and two integers
$\delta\in\{0,1\}$ and $m\in\{2,3\}$ such that.
\[
\mathcal{A}\overset{\mathrm{{\text{ {\tiny $k$-alg.}}}}}{\simeq}L_{1}%
\times\cdots\times L_{n}\times(\mathrm{k}[T]/T^{m})^{\delta}%
\]

\end{corollary}

\begin{proof}
Thanks to Corollary~\ref{cor-monogene}, we get $\mathcal{A}=\mathrm{k}[T]/P$.
Let us write the irreducible decomposition of $P$
\[
P=\prod_{i=1}^{s}{P_{i}}^{n_{i}}%
\]
with $P_{i}\in\mathrm{k}[T]$ irreducible, $n_{i}>0$ and $P_{i}$ and $P_{j}$
non associated for $i\neq j$. Chinese Reminder theorem tells us that
\[
\mathcal{A}\overset{\mathrm{{\text{ {\tiny $k$-alg.}}}}}{\simeq}%
\mathrm{k}[T]/{P_{1}}^{n_{1}}\times\cdots\times\mathrm{k}[T]/{P_{s}}^{n_{s}%
}\,.
\]
In particular, $\mathrm{k}[T]/{P_{i}}^{n_{i}}$ is a quotient of $\mathcal{A}$.

So Remark~\ref{rem-quotient} and Lemma~\ref{lem-deg2} ensures us that
$n_{i}=1$ if $\deg P_{i}\geq2$. In this case $L_{i}=\mathrm{k}[T]/P_{i}$ is a
finite extension of $\mathrm{k}$ generated by one element.

If $\deg P_{i}=1$, then $P_{i}=T-\lambda$ for some $\lambda\in\mathrm{k}$ and
$\mathrm{k}[T]/{P_{i}}^{n_{i}}$ is isomorphic to $\mathrm{k}[T]/T^{n_{i}}$.
Lemma~\ref{lem-kxxn} ensures us that $n_{i}\in\{1,2,3\}$. If $n_{i}=1$ then we
get $\mathrm{k}$. Let us consider the case where $n_{i}\in\{2,3\}$.
Lemma~\ref{lem-nilpotent} tells us that there is at most one factor of this
type and we get the structure result.
\end{proof}

We have just shown that algebras with a finite number of subalgebras have a
certain form (a finite product of fields generated by one element with
possibly a $k[T]/T^{2}$ or $k[T]/T^{3}$ factor). Our aim is now to show that
algebras with this given form have only a finite number of subalgebras. To
prove this, we first show that we can restrict our attention to subalgebras
generated by one element (Lemma~\ref{lem-finite-monogene}) and then give a
description of all the subalgebras generated by one element of such an algebra
(Proposition~\ref{prop-subalgebra-product}).

\begin{lemma}
[Subalgebra generated by one element]\label{lem-finite-monogene} Let
$\mathcal{A}$ be a finite-dimensional algebra over $\mathrm{k}$. Then
$\mathcal{A}$ has only a finite number of subalgebras if and only if
$\mathcal{A}$ has only a finite number of subalgebras generated by one element.
\end{lemma}

\begin{proof}
The part only if is clear. Let us suppose that $\mathcal{A}$ has only a finite
number of subalgebras generated by one element. Let $\mathcal{B}$ be a
subalgebra of $\mathcal{A}$. We have $\mathcal{B}=\cup_{y\in\mathcal{B}%
}\mathrm{k}[y]$. So $\mathcal{B}$ is a union of subalgebras generated by one
element. There is only a finite number of such subalgebras since there is only
a finite number of subalgebras $\mathrm{k}[y]$.
\end{proof}

Let us now determine subalgebras generated by one element of algebras of the
form $L_{1}\times\cdots\times L_{n}\times\mathrm{k}[T]/T^{m}$ where
$n,m\in\mathbb{N}$, $L_{i}$ is an algebraic field extension of $\mathrm{k}$.

\begin{proposition}
\label{prop-subalgebra-product} Let $m,n\in\mathbb{N}$. For $i\in
\lbrack\,1,m\,]$, let $L_{i}$ be an algebraic field extension of $\mathrm{k}$.
Set $\mathcal{A}=L_{1}\times\cdots\times L_{n}\times\mathrm{k}[T]/T^{m}$. For
simplicity write $L_{n+1}$ for $\mathrm{k}[T]/T^{m}$ and for any $i\in
\lbrack1,n+1]$, let $p_{i}:\mathcal{A}\rightarrow L_{i}$ be the projection on
the $i^{\text{th}}$ factor.

Let $\mathcal{B}$ be a subalgebra of $\mathcal{A}$ generated by one element
and for $i\in\lbrack1\,,n+1]$, $K_{i}=p_{i}(\mathcal{B})\subset L_{i}$. There exists

\begin{itemize}
\item[$(i)$] a partition of $[\,1,n+1]$%
\[
\lbrack1\,,\,n+1]=I_{1}\bigsqcup\cdots\bigsqcup I_{r}%
\]
with $n+1\in I_{r}$ if $m\neq0$.

\item[$(ii)$] a family of integers $(i_{1},\ldots,i_{r})\in I_{1}\times
\cdots\times I_{r}$ with $i_{r}=n+1$ if $m\neq0$. For every $j\in
\lbrack1\,,\,r]$, let us write $I_{j}=\{i_{j},u_{j,1},\ldots,u_{j,s_{j}}\}$.

\item[$(iii)$] For every $j\in\lbrack1\,,\,r]$ and every $\ell\in
\lbrack1\,,\,s_{j}]$, $\mathrm{k}$-algebras morphisms $\sigma_{j,\ell
}:K_{i_{j}}\rightarrow K_{u_{j,\ell}}$
\end{itemize}

such that after reordering the factors of the product, we have
\[
\scriptstyle\mathcal{B}=\{(x_{1},\sigma_{1,1}(x_{1}),\ldots,\sigma_{1,s_{1}%
}(x_{1}),x_{2},\ldots,\sigma_{2,s_{2}}(x_{2}),\ldots,x_{r},\ldots
,\sigma_{r,s_{r}}(x_{r})),\quad x_{j}\in K_{i_{j}}\text{ for }j\in
\lbrack1\,,\,r]\}\,.
\]

\end{proposition}

\begin{proof}
Let us start by giving an overview of the proof. Let $y\in\mathcal{A}$ such
that $\mathcal{B}=\mathrm{k}[y]$. We write $y=(y_{1},\ldots,y_{n+1})$ with
$y_{i}\in L_{i}$ for all $i$. The partition we are looking for is in fact
given by gathering together the $y_{i}$ with the same minimal polynomial
(except for $y_{n+1}$ which may play a special role). After this, we link
$y_{i}$ and $y_{j}$ with the same minimal polynomial through a morphism of
algebras. We finally get the independence of blocks with different minimal
polynomial using the Chinese remainder theorem.

Let us now begin the proof. We have $\mathrm{k}[y]=\{(Q(y_{1}),\ldots
,Q(y_{n+1})),\ Q\in\mathrm{k}[T]\}$ and then $K_{i}=\mathrm{k}[y_{i}]$. We
define the equivalence relation $\sim$ on $[1\,,\,n]$ by $i\sim j$ if $y_{i}$
and $y_{j}$ have the same minimal polynomial over $\mathrm{k}$. This defines a
partition of $[1\,,\,n]=J_{1}\sqcup\cdots\sqcup J_{s}$. Let us now consider
the index $n+1$. For $i=n+1$ (if $m\neq0$), the minimal polynomial of
$y_{n+1}$ is of the form $(T-\lambda)^{s}$ for some $\lambda\in\mathrm{k}$ and
$s\leq m$. If there exists $i\in\lbrack1\,,\,n]$ such that $y_{i}=\lambda
\in\mathrm{k}$ then we add $n+1$ to the equivalence class of $i$ and obtain a
partition of $[1\,,\,n]$ in $s$ parts. In this case, we set $r=s$ and we
number these parts such that $n+1$ is in the part indexed by $r$. If there
does not exist an $i$ such that $y_{i}=\lambda$, then we set $r=s+1$ and add
the part $\{n+1\}$ to the partition of $[1\,,\,n]$ to get the desired
partition of $[1\,,\,n+1]$.

Finally, we write the partition of $[1\,,\,n+1]$ we just obtain:
\[
\lbrack1\,,\,n+1]=I_{1}\bigsqcup\cdots\bigsqcup I_{r}%
\]

For each part of the partition, we choose a representative $i_{j}$ of the
subset. If $n+1$ is not alone in his part, then we choose $n+1$ to be the
representative of this subset.


First, assume that $m=0$. For $j\in\lbrack1\,,\,r]$ and $\alpha\in I_{j}$ the
elements $y_{\alpha}$ and $y_{i_{j}}$ have the same minimal polynomial so
there exists an isomorphism $\sigma$ of extensions from $\mathrm{k}[y_{i_{j}%
}]=K_{i_{j}}$ to $\mathrm{k}[y_{\alpha}]=K_{\alpha}$ sending $y_{i_{j}}$ to
$y_{\alpha}$. In particular, we have $\sigma(Q(y_{i_{j}}))=Q(y_{\alpha})$ for
all $Q\in\mathrm{k}[T]$.

To get the desired description of $\mathcal{B}$, it suffices now to find
$Q_{j}\in\mathrm{k}[T]$ such that $Q_{j}(y_{i_{j}})=1$ and $Q_{j}(y_{i_{\ell}%
})=0$ for all $\ell\neq j$. This is possible since the minimal polynomial
$P_{\ell}$ of the $y_{i_{\ell}}$ are prime to each other (they are irreducible
and distinct): we write
\[
1=UP_{j}+V\prod_{\ell\neq j}P_{\ell}%
\]
and consider $Q_{j}=V\prod_{\ell\neq j}P_{\ell}$.

Let us now assume that $m\neq0$. For $j\in\lbrack1\,,\,r-1]$, there is no
difference with the preceding case. For $j=r$, we have to be more careful: for
$\alpha\neq n+1\in I_{r}$, we can define the following morphisms of algebras
\[
K_{n+1}=\mathrm{k}[y_{n+1}]\overset{\mathrm{{\text{ {\tiny $k$-alg.}}}}%
}{\simeq}\mathrm{k}[T]/(T-\lambda)^{s}\rightarrow\mathrm{k}[T]/(T-\lambda
)\overset{\mathrm{{\text{ {\tiny $k$-alg.}}}}}{\simeq}\mathrm{k}=K_{\alpha}%
\]
where the first isomorphism sends $y_{n+1}$ to the class of $T$ and the second
isomorphism sends the class of $T$ to $\lambda\in\mathrm{k}=K_{\alpha}$. So by
composition, we get the desired morphism of algebras.

Finally, to get the desired description of $\mathcal{B}$ in this case, it
suffices to adapt the Chinese remainder argument. For this, we remark that
$(T-\lambda)^{s}$ is prime with every irreducible polynomial over $\mathrm{k}$
except $T-\lambda$. But the index $i\in\lbrack1\,,\,n]$ such that $T-\lambda$
is the minimal polynomial $y_{i}$ are precisely in $I_{r}$.
\end{proof}

\begin{corollary}
\label{cor-rec-struct} Let $L_{1},\ldots,L_{n}$ be finite field extensions of
$\mathrm{k}$ generated over $\mathrm{k}$ by one element. Consider also two
integers $\delta\in\{0,1\}$ and $m\in\{2,3\}$. Then
\[
\mathcal{A}=L_{1}\times\cdots\times L_{n}\times(\mathrm{k}[T]/T^{m})^{\delta}%
\]
has only a finite number of subalgebras.
\end{corollary}

\begin{proof}
Lemma~\ref{lem-finite-monogene} shows that it suffices to prove that
$\mathcal{A}$ has only a finite number of subalgebras generated by one
element. But proposition~\ref{prop-subalgebra-product} implies that a
subalgebra of $\mathcal{A}$ generated by one element is determined by a family
of subalgebras of the $L_{i}$ and by algebra morphisms between them. But each
$L_{i}$ has only a finite number of subalgebras and moreover Dedekind Lemma
(\cite{bourbaki}) ensures us that there exists only a finite number of
$\mathrm{k}$-algebra morphisms with values in a finite extension of
$\mathrm{k}$. So there is only a finite number of such subalgebras.
\end{proof}

Finally we get the structure theorem for algebras with only a finite number of subalgebras.

\begin{theorem}
\label{TH_classification}Let $\mathcal{A}$ be a finite-dimensional algebra
over $\mathrm{k}$. Then the following statements are equivalent.

\begin{enumerate}
\item $\mathcal{A}$ has only a finite number of subalgebras.

\item There exists finite algebraic extensions $L_{1},\ldots,L_{n}$ generated
over $\mathrm{k}$ by one element and two integers $\delta\in\{0,1\}$ and
$m\in\{2,3\}$ such that
\begin{equation}
\mathcal{A}\overset{\mathrm{{\text{ {\tiny $k$-alg.}}}}}{\simeq}L_{1}%
\times\cdots\times L_{n}\times(\mathrm{k}[T]/T^{m})^{\delta}. \label{form}%
\end{equation}

\item There exists $g\in\mathcal{A}$ such that $\mathcal{A}=\mathrm{k}[g]$
where the minimal polynomial $\mu_{g}$ of $g$ has the form%
\[
\mu_{g}=P_{1}\cdots P_{n}Q^{m\delta}%
\]
where $P_{1},\ldots,P_{n},Q$ are distinct irreducible polynomials in
$\mathrm{k}[T],$ $\delta\in\{0,1\}$, $m\in\{2,3\}$ and $\deg Q=1$.
\end{enumerate}
\end{theorem}

\begin{proof}
The theorem follows from~Corollary~\ref{cor-structure} and
Corollary~\ref{cor-rec-struct}.
\end{proof}

\begin{remark}
Consider an infinite-dimensional commutative algebra $\mathcal{A}$ with a
finite number of finite-dimensional subalgebras $\mathcal{B}_{1}%
,\ldots,\mathcal{B}_{r}$.\ There exist $g_{1},\ldots,g_{r}$ in $\mathcal{A}$
such that $\mathcal{B}_{i}=\mathrm{k}[g_{i}]$ for any $i=1,\ldots,r$.\ Then
$\mathcal{B}=\mathrm{k}[g_{1},\ldots,g_{r}]$ is a finite-dimensional
subalgebra of $\mathcal{A}$ containing the subalgebras $\mathcal{B}_{1}%
,\ldots,\mathcal{B}_{r}$. Therefore $\mathcal{B}$ coincides in fact with one
of the algebras $\mathcal{B}_{i}$.\ This means that all the finite-dimensional
subalgebras of $\mathcal{A}$ appear as subalgebras of the finite-dimensional
algebra $\mathcal{B}\subset\mathcal{A}$ and we have a structure theorem for
$\mathcal{B}$.
\end{remark}

\begin{remark}
Consider a $k$-algebra $\mathcal{A}$ with a finite number of subalgebras. Then
$\mathcal{A}$ is finite-dimensional over $k$. Indeed, for $y\in\mathcal{A}$,
$k[y]$ is finite-dimensional. Otherwise, $k[y]$ would be isomorphic to $k[T]$
and would have infinitely many subalgebras. Moreover, there are finitely many
subalgebras of the form $k[y]$. Write these algebras $k[y_{1}],\ldots
,k[y_{r}]$. We then have
\[
\mathcal{A}=\bigcup_{i=1}^{r}k[y_{i}]=\sum_{i=1}^{r}k[y_{i}]
\]
since for each $y\in\mathcal{A}$, the algebra $k[y]$ coincides with one of the
algebras $k[y_{1}],\ldots,k[y_{r}]$. This equality shows that $\mathcal{A}$ is finite-dimensional.
\end{remark}

\subsection{Some examples}

\label{Subsec_Examples}

\subsubsection{Algebras of functions defined on a finite set}

Let $S=\{s_{1},\ldots,s_{n}\}$ be a finite set and write $\mathcal{F}_{S}$ for
the algebra of functions $f:S\rightarrow\mathrm{k}$.\ The algebra
$\mathcal{F}_{S}$ is clearly isomorphic to $\mathrm{k}^{n}$. Thus, applying
Theorem~\ref{TH_classification} to $\mathcal{F}_{S}$ allows us to recover the
very classical following fact : $\mathcal{F}_{S}$ admits a finite number of
subalgebras parametrized by the partitions $S=\bigsqcup\limits_{i=1}^{m}S_{m}$
of $S$. The subalgebra $\mathcal{F}_{S_{1},\ldots,S_{m}}$ associated to such a
partition is the algebra of functions $f\in\mathcal{F}_{S}$ which are constant
on each set $S_{i},i=1,\ldots,m$.\ Observe also we have $\mathcal{F}%
_{S}=\mathrm{k}[f]$ for any function $f$ such that $f(s_{i})\neq f(s_{j})$ for
any $i\neq j$.

A special case is the algebra $\mathcal{F}_{G}$ of complex central functions
defined on a group $G$. Here, $S$ is the set of conjugacy classes of $G$. In
particular the characters of $G$ belongs to $\mathcal{F}_{G}$.\ Let us
consider $A=\{\chi_{0}=1,\chi_{1},\ldots,\chi_{r}\}$ and $B=\{\varphi
_{0}=1,\varphi_{1},\ldots,\varphi_{s}\}$ two subsets of irreducible characters
corresponding to the irreducible representations $U_{0}=\mathrm{k}%
,U_{1},\ldots,U_{r}$ and $V_{0}=\mathrm{k},V_{1},\ldots,V_{s},$ respectively.
Recall that for any $(i,j)\in\{0,\ldots,r\}\times\{0,\ldots,s\}$, $\chi
_{i}\varphi_{j}$ is the characters of the tensor product $U_{i}\otimes V_{j}$.
Let $\mathcal{V}$ be the set of characters of the representations obtained as
direct sums of some copies of the $U_{i}\otimes V_{j}$'s. Observe
$\mathcal{V}$ is not a $\mathrm{k}$-space since it only contains linear
combinations of the $\chi_{i}\varphi_{j}$ with nonnegative integer
coefficients. Nevertheless, we have $\mathrm{k}\langle\mathcal{V}%
\rangle=\mathrm{k}\langle AB\rangle$ since the family $\{\chi_{i}\varphi
_{j}\mid(i,j)\in\{0,\ldots,r\}\times\{0,\ldots,s\}\}$ generates $\mathrm{k}%
\langle AB\rangle$ as a $\mathrm{k}$-space. In particular, $\mathrm{k}\langle
AB\rangle$ admits a basis of characters in $\mathcal{V}$. It follows that
$\dim_{\mathrm{k}}(AB)$ is the maximal number of linearly independent
characters in $\mathcal{V}$. When $G$ is abelian or the characters $\chi_{i}$
and $\varphi_{j}$ are linear, $\dim_{\mathrm{k}}(AB)$ is simply the
cardinality of $\{\chi_{i}\varphi_{j}\mid(i,j)\in\{0,\ldots,r\}\times
\{0,\ldots,s\}\}$ since distinct linear characters are always independent.

\subsubsection{Matrix subalgebras with finitely many subalgebras}

It is easy to construct subalgebras of matrix algebras with the form
(\ref{form}). Indeed consider irreducible polynomials $P_{1},\ldots,P_{n},Q$
and integers $m\in\{2,3\}$ and $\delta\in\{0,1\}$ as in Theorem
\ref{TH_classification} and their companion matrices $\mathcal{C}_{P_{1}%
},\ldots,\mathcal{C}_{P_{n}},\mathcal{C}_{Q^{m}}$. Set
\[
r=\deg(P_{1})+\cdots+\deg(P_{n})+m\delta\deg(Q)
\]
and define $M$ as the $r\times r$ matrix with coefficients in $\mathrm{k}$
obtained as the block diagonal matrix with blocs $\mathcal{C}_{P_{1}}%
,\ldots,\mathcal{C}_{P_{n}},\mathcal{C}_{Q^{m}}$ when $\delta=1$ and
$\mathcal{C}_{P_{1}},\ldots,\mathcal{C}_{P_{n}}$ when $\delta=0$. Then the
subalgebra $\mathrm{k}[M]$ of $\mathcal{M}_{r}(\mathrm{k})$ has the form
(\ref{form}).

\subsubsection{Algebras of complex valued continuous functions on a connected
space}

Consider $I$ a connected space and define $\mathcal{C}_{I}$ as the
$\mathbb{C}$-algebra of continuous functions $f:I\rightarrow\mathbb{C}$%
.$\ $Then, the unique finite-dimensional subalgebra of $\mathcal{C}_{I}$ is
that of constant functions. Indeed if we consider $\mathcal{A}$ such a
subalgebra and $f\in\mathcal{A}$, then the minimal polynomial $\mu_{f}$ is
such that $\mu_{f}(f)(x)=0$ for any $x\in I$. Hence all the values of the
connected set $\operatorname{Im}f$ are zeroes for $\mu_{f}$. Since the only
finite connected sets of $\mathbb{C}$ are singletons, the set
$\operatorname{Im}f$ is reduced to a point and $f$ is constant.

\section{Kneser type theorems}

\label{Sec_Kneser}In this Section, we state an analogue of Kneser's theorem
for algebras.

\subsection{Kneser-Diderrich theorem for a wide class of algebras}

In this section, $\mathcal{A}$ satisfies \textbf{H}$_{\mathrm{s}}$ or
\textbf{H}$_{\mathrm{w}}$. Let us consider a finite nonempty subset $A$ of
$\mathcal{A}_{\ast}$.\ We say that $A$ is commutative when $aa^{\prime
}=a^{\prime}a$ for any $a,a^{\prime}\in A$. This then implies that the
elements of $\mathrm{k}\langle A\rangle$ are pairwise commutative.\ Moreover
the algebra $\mathbb{A}(A)$ generated by $A$ is then commutative. Typical
examples of commutative sets are geometric progressions $A=\{a^{r}%
,a^{r+1},\ldots,a^{r+s}\}$ with $r$ and $s$ integers. The following theorem is
an analogue, for algebras, of a theorem by Diderrich \cite{Die} extending
Kneser's theorem for arbitrary groups when only the subset $A$ is assumed commutative.

\begin{theorem}
\label{Th_KL_Acom} Assume $\mathcal{A}$ satisfies \textbf{H}$_{\mathrm{w}}$
and consider $A$ and $B$ be two finite nonempty subsets of $\mathcal{A}_{\ast
}$ such that $\mathrm{k}\langle A\rangle\cap U(\mathcal{A})\neq\emptyset$ and
$\mathrm{k}\langle B\rangle\cap U(\mathcal{A})\neq\emptyset$. Assume that $A$
is commutative and $\mathbb{A}(A)$ admits a finite number of
finite-dimensional subalgebras. Let $\mathcal{H}:=\mathcal{H}_{l}(AB).$

\begin{enumerate}
\item We have
\[
\dim_{\mathrm{k}}(AB)\geq\dim_{\mathrm{k}}(A)+\dim_{\mathrm{k}}(B)-\dim
(\mathcal{H})
\]
In particular, if $AB$ is not left periodic
\[
\dim_{\mathrm{k}}(AB)\geq\dim_{\mathrm{k}}(A)+\dim_{\mathrm{k}}(B)-1
\]

\item If $\mathcal{A}$ is commutative, then
\[
\dim_{\mathrm{k}}(AB)\geq\dim_{\mathrm{k}}(\mathcal{H}A)+\dim_{\mathrm{k}%
}(\mathcal{H}B)-\dim_{\mathrm{k}}(\mathcal{H})\geq\dim_{\mathrm{k}}%
(A)+\dim_{\mathrm{k}}(B)-\dim_{\mathrm{k}}(\mathcal{H})
\]

\end{enumerate}
\end{theorem}

\smallskip To prove this theorem we need the following preparatory lemma.

\begin{lemma}
\label{Lem_diff}Assume $\mathcal{A}$ satisfies \textbf{H}$_{\mathrm{s}}$. Let
$A$ and $B$ be two finite subsets of $\mathcal{A}_{\ast}$ such that $A$ is
commutative, $\mathrm{k}\langle A\rangle\cap U(\mathcal{A})\neq\emptyset$ and
$\mathrm{k}\langle B\rangle\cap U(\mathcal{A})\neq\emptyset$. Then, for each
$a\in\mathrm{k}\langle A\rangle\cap U(\mathcal{A})$, there exists a
(commutative) finite-dimensional subalgebra $\mathcal{A}_{a}$ of $\mathcal{A}$
such that $\mathrm{k}\subseteq\mathcal{A}_{a}\subseteq\mathbb{A}(A)$ and a
vector space $V_{a}$ contained in $\mathrm{k}\langle AB\rangle$ such that
$V_{a}\cap U(\mathcal{A})\neq\emptyset$, $\mathcal{A}_{a}V_{a}=V_{a},$
$\mathrm{k}\langle aB\rangle\subseteq V_{a}$ and
\[
\dim_{\mathrm{k}}(V_{a})+\dim_{\mathrm{k}}(\mathcal{A}_{a})\geq\dim
_{\mathrm{k}}(A)+\dim_{\mathrm{k}}(B).
\]

\end{lemma}

\begin{proof}
The hypothesis on a $\mathcal{A}$ ensures that each subspace of $\mathcal{A}$
containing an invertible element admits a basis of invertible elements
(see~Proposition~\ref{Prop_base_inv}).

By replacing $A$ by $A^{\prime}=a^{-1}A$ with $a\in\mathrm{k}\langle
A\rangle\cap U(\mathcal{A})\neq\emptyset$, we can establish the lemma only for
$a=1$. Indeed, if there exist a subalgebra $\mathcal{B}\subseteq
\mathbb{A}(A^{\prime})$ and a vector space $V\neq\{0\}$ contained in
$\mathrm{k}\langle A^{\prime}B\rangle$ such that $\mathcal{B}V=V$ and
$\mathrm{k}\langle B\rangle\subseteq V$ with
\[
\dim_{\mathrm{k}}(V)+\dim_{\mathrm{k}}(\mathcal{B})\geq\dim_{\mathrm{k}%
}(A^{\prime})+\dim_{\mathrm{k}}(B),
\]
it suffices to take $V_{a}=aV$ and $\mathcal{A}_{a}=\mathcal{B}\subseteq
\mathbb{A}(A^{\prime})\subseteq\mathbb{A}(A)$.\ Since $\mathcal{B}%
\subseteq\mathbb{A}(A),$ we must have $\mathcal{B}a=a\mathcal{B}$ for any
$a\in A$ and $\mathcal{B}(V_{a})=\mathcal{B}(aV)=a\mathcal{B}V=aV=V_{a}$.
Moreover $\mathrm{k}\langle aB\rangle=a\mathrm{k}\langle B\rangle\subseteq
aV=V_{a}$ and $\dim_{\mathrm{k}}(V_{a})+\dim_{\mathrm{k}}(\mathcal{A}_{a}%
)\geq\dim_{\mathrm{k}}(A)+\dim_{\mathrm{k}}(B)$ because $\dim_{\mathrm{k}%
}(V_{a})=\dim_{\mathrm{k}}(V)$ and $\mathcal{A}_{a}=\mathcal{B}$.

We can also assume that $1\in B$ by replacing $B$ by $B^{\prime}=Bb^{-1}$ with
$b\in\mathrm{k}\langle B\rangle\cap U(\mathcal{A})\neq\emptyset$. Indeed, if
there exist a subalgebra $\mathcal{B}^{\prime}\subseteq\mathbb{A}(A)$ and a
vector space $V^{\prime}\neq\{0\}$ contained in $\mathrm{k}\langle AB^{\prime
}\rangle$ such that $\mathcal{B}^{\prime}V^{\prime}=V^{\prime}$ and
$\mathrm{k}\langle B^{\prime}\rangle\subseteq V^{\prime}$ with
\[
\dim_{\mathrm{k}}(V^{\prime})+\dim_{\mathrm{k}}(\mathcal{B}^{\prime})\geq
\dim_{\mathrm{k}}(A)+\dim_{\mathrm{k}}(B^{\prime}),
\]
it suffices to take $V=V^{\prime}b$ and $\mathcal{A}_{a}=\mathcal{B}^{\prime}%
$.\ We will have then $V=V^{\prime}b\subseteq\mathrm{k}\langle AB^{\prime
}\rangle b=\mathrm{k}\langle AB\rangle,$ $\mathcal{A}_{a}V=\mathcal{B}%
^{\prime}(V^{\prime}b)=(\mathcal{B}^{\prime}V^{\prime})b=V^{\prime}b=V,$
$\mathrm{k}\langle B\rangle=\mathrm{k}\langle B^{\prime}\rangle b\subseteq
V^{\prime}b=V$ and
\[
\dim_{\mathrm{k}}(V)+\dim_{\mathrm{k}}(\mathcal{A}_{a})\geq\dim_{\mathrm{k}%
}(A)+\dim_{\mathrm{k}}(B)
\]
since $\dim_{\mathrm{k}}(B)=\dim_{\mathrm{k}}(B^{\prime})$ and $\dim
_{\mathrm{k}}(V)=\dim_{\mathrm{k}}(V^{\prime})$.

We thus assume in the remainder of the proof that $1\in A\cap B$ and proceed
by induction on $\dim_{\mathrm{k}}(A).$ When dim$_{\mathrm{k}}(A)=1$, we have
$\mathrm{k}\langle A\rangle=\mathrm{k}=\mathbb{A}(A)$. It suffices to take
$V_{1}=V=\mathrm{k}\langle B\rangle$ (with $1\in B$) and $\mathcal{A}%
_{1}=\mathrm{k}=\mathbb{A}(A)$. Assume $\dim_{\mathrm{k}}(A)>1$.\ Given
$e\in\mathrm{k}\langle B\rangle\cap U(\mathcal{A})$, define $A(e)$ and $B(e)$
to be finite subsets of $\mathcal{A}_{\ast}$ such that
\[
\mathrm{k}\langle A(e)\rangle=\mathrm{k}\langle A\rangle\cap\mathrm{k}\langle
B\rangle e^{-1}\quad\text{ and }\quad\mathrm{k}\langle B(e)\rangle
=\mathrm{k}\langle B\rangle+\mathrm{k}\langle A\rangle e.
\]
Observe that $\mathrm{k}\langle A(e)\rangle$ and $\mathrm{k}\langle
B(e)\rangle$ contain $\mathrm{k}$ since $1\in A\cap B$. Thus we may and do
assume that $1\in A(e)\cap B(e)$.\ Moreover, $\mathrm{k}\langle A(e)\rangle
\mathrm{k}\langle B(e)\rangle$ is contained in $\mathrm{k}\langle AB\rangle$.
Indeed, for $v\in\mathrm{k}\langle A\rangle\cap\mathrm{k}\langle B\rangle
e^{-1}$ and $w\in\mathrm{k}\langle B\rangle,$ we have $vw\in\mathrm{k}\langle
A\rangle\mathrm{k}\langle B\rangle\subseteq\mathrm{k}\langle AB\rangle$
because $v\in\mathrm{k}\langle A\rangle$. Set $v=ze^{-1}$ with $z\in
\mathrm{k}\langle B\rangle$. If $w\in\mathrm{k}\langle A\rangle e,$ we have
$vw\in ze^{-1}\mathrm{k}\langle A\rangle e$. But $ze^{-1}\in\mathrm{k}\langle
A\rangle$ and $A$ is commutative. Therefore, $vw\in\mathrm{k}\langle A\rangle
ze^{-1}e=\mathrm{k}\langle A\rangle z\subseteq\mathrm{k}\langle A\rangle
\mathrm{k}\langle B\rangle\subseteq\mathrm{k}\langle AB\rangle$.\ In
particular, $\dim_{\mathrm{k}}(A(e)B(e))\leq\dim_{\mathrm{k}}(AB)$.\ We get%
\begin{multline*}
\dim_{\mathrm{k}}(A(e))+\dim_{\mathrm{k}}(B(e))=\dim_{\mathrm{k}}%
(\mathrm{k}\langle A\rangle\cap\mathrm{k}\langle B\rangle e^{-1}%
)+\dim_{\mathrm{k}}(\mathrm{k}\langle B\rangle+\mathrm{k}\langle A\rangle
e)=\\
\dim_{\mathrm{k}}(\mathrm{k}\langle A\rangle e\cap\mathrm{k}\langle
B\rangle)+\dim_{\mathrm{k}}(\mathrm{k}\langle B\rangle+\mathrm{k}\langle
A\rangle e)=\dim_{\mathrm{k}}(Ae)+\dim_{\mathrm{k}}(B)=\dim_{\mathrm{k}%
}(A)+\dim_{\mathrm{k}}(B).
\end{multline*}
Also $A(e)\subseteq\mathrm{k}\langle A\rangle$.\ 

Assume $\mathrm{k}\langle A(e)\rangle=\mathrm{k}\langle A\rangle$ for any
$e\in\mathrm{k}\langle B\rangle\cap U(\mathcal{A})$.\ Then $\mathrm{k}\langle
A\rangle e\subseteq\mathrm{k}\langle B\rangle$ for any $e\in\mathrm{k}\langle
B\rangle\cap U(\mathcal{A})$. Thus $\mathrm{k}\langle AB\rangle\subseteq
\mathrm{k}\langle B\rangle$ by Proposition \ref{Prop_base_inv}.\ Indeed
$\mathrm{k}\langle B\rangle$ admits a basis contained in $U(\mathcal{A})$ and
the products $xy$ with $x\in A$ and $y\in\mathrm{k}\langle B\rangle\cap
U(\mathcal{A})$ generate $\mathrm{k}\langle AB\rangle$.\ Since $1\in A$, we
have in fact $\mathrm{k}\langle AB\rangle=\mathrm{k}\langle B\rangle$.\ The
subalgebra $\mathcal{A}_{1}=\mathbb{A}(A)$ is commutative. Take $V_{1}%
=\mathrm{k}\langle B\rangle$ (with $1\in B$). Then $\mathcal{A}_{1}V_{1}%
=V_{1}$ since $\mathrm{k}\langle AB\rangle\subseteq\mathrm{k}\langle B\rangle
$.\ In particular, $\mathcal{A}_{1}$ is finite-dimensional since $1 \in V_{1}%
$.\ We clearly have $V_{1}=\mathrm{k}\langle AB\rangle$ and $\mathrm{k}\langle
B\rangle\subseteq\mathrm{k}\langle AB\rangle=V_{1}$ as desired. We also get%
\[
\dim_{\mathrm{k}}(V_{1})+\dim_{\mathrm{k}}(\mathcal{A}_{1})\geq\dim
_{\mathrm{k}}(A)+\dim_{\mathrm{k}}(B).
\]

Now assume $\mathrm{k}\langle A(e)\rangle\neq\mathrm{k}\langle A\rangle$ for
at least one $e\in\mathrm{k}\langle B\rangle\cap U(\mathcal{A})$.\ Then
$0<\dim_{\mathrm{k}}(A(e))<\dim_{\mathrm{k}}(A)$ and $1\in A(e)\cap B(e)$.\ By
our induction hypothesis, there exists a finite-dimensional subalgebra
$\mathcal{A}_{1}$ of $\mathbb{A}(A(e))\subseteq\mathbb{A}(A)$ and a nonzero
$\mathrm{k}$-vector space $V_{1}\subseteq\mathrm{k}\langle A(e)B(e)\rangle
\subseteq\mathrm{k}\langle AB\rangle$, such that $V_{1}\cap U(\mathcal{A}%
)\neq\emptyset,\mathcal{A}_{1}V_{1}=V_{1}$ and $\mathrm{k}\langle
B\rangle\subseteq\mathrm{k}\langle B(e)\rangle\subseteq V_{1}$ with
\[
\dim_{\mathrm{k}}(V_{1})+\dim_{\mathrm{k}}(\mathcal{A}_{1})\geq\dim
_{\mathrm{k}}(A(e))+\dim_{\mathrm{k}}(B(e))=\dim_{\mathrm{k}}(A)+\dim
_{\mathrm{k}}(B).
\]
The subalgebra $\mathcal{A}_{1}\subseteq\mathbb{A}(A)$ and the nonzero space
$V_{1}\supset\mathrm{k}\langle B\rangle$ satisfy the statement of the lemma
for the pair of subsets $A$ and $B$ which concludes the proof.
\end{proof}

\medskip

We are now ready to prove \ref{Th_KL_Acom}.

\begin{proof}
(of Theorem \ref{Th_KL_Acom})

We first remark that if $\mathcal{A}$ is a subalgebra of an algebra
$\mathcal{B}$ then the stabilizer of $k\langle AB \rangle$ in $\mathcal{B}$ is
the stabilizer of $k\langle AB \rangle$ in $\mathcal{A}$. Indeed, for $a \in A
\cap U(\mathcal{A})$, $b \in B \cap U(\mathcal{A})$ and $x$ in the stabilizer
of $k\langle AB \rangle$ in $\mathcal{B}$, we have $xab \in k\langle AB
\rangle$ and so $x \in k\langle AB \rangle b^{-1}a^{-1} \subset\mathcal{A}$.

Since $A$ stays commutative in $\mathcal{B}$ and $\mathbb{A}(A)\subset
\mathcal{A}\subset\mathcal{B}$ has also a finite number of finite dimensional
subalgebras, it suffices to prove our theorem when $\mathcal{A}$ satisfies
\textbf{H}$_{\mathrm{s}}$.

1: Let $\{x_{1},\ldots,x_{n}\}$ be a basis of $\mathrm{k}\langle A\rangle$
with $x_{1}$ invertible.\ For any $\alpha\in\mathrm{k}$, set $x_{\alpha}%
=x_{1}+\alpha x_{2}+\cdots+\alpha^{n-1}x_{n}$.\ Assume $x_{\alpha}$ is
invertible. Since $\mathrm{k}$ is infinite and by Lemma \ref{Lem_diff}, there
exists a finite-dimensional subalgebra $\mathcal{A}_{\alpha}$ such that
$\mathrm{k}\subseteq\mathcal{A}_{\alpha}\subseteq\mathbb{A}(A)\subseteq
\mathcal{A}$ and a $\mathrm{k}$-vector space $V_{\alpha}\subseteq
\mathrm{k}\langle AB\rangle$ with $x_{\alpha}B\subseteq V_{\alpha},$
$\mathcal{A}_{\alpha}V_{\alpha}=V_{\alpha},$ $V_{\alpha}\cap U(\mathcal{A}%
)\neq\emptyset$ and
\begin{equation}
\dim_{\mathrm{k}}(V_{\alpha})+\dim_{\mathrm{k}}(\mathcal{A}_{\alpha})\geq
\dim_{\mathrm{k}}(A)+\dim_{\mathrm{k}}(B). \label{ineqK}%
\end{equation}
Since
$\mathcal{A}_{\alpha}\subseteq\mathbb{A}(A)$ and $\mathbb{A}(A)$ is
commutative and admits a finite number of finite-dimensional subalgebras
containing $\mathrm{k,}$ there should exist by Lemma \ref{Lemma_VDM} $n$
distinct scalars $\alpha_{1},\ldots,\alpha_{n}$ in $\mathrm{k}$ such that
\[
\mathcal{A}_{\alpha_{1}}=\mathcal{A}_{\alpha_{2}}=\cdots=\mathcal{A}%
_{\alpha_{n}}=\mathcal{B}%
\]
and $x_{\alpha_{1}},\ldots,x_{\alpha_{n}}$ form a basis of invertible elements
of $\mathrm{k}\langle A\rangle$ over $\mathrm{k}$.\ We thus have
$\mathrm{k}\langle AB\rangle=\sum_{i=1}^{n}x_{\alpha_{i}}\mathrm{k}\langle
B\rangle\subseteq\sum_{i=1}^{n}V_{\alpha_{i}}$ since $x_{\alpha_{i}}%
\mathrm{k}\langle B\rangle\subseteq V_{\alpha_{i}}$ for any $i=1,\ldots
,n$.\ On the other hand, $V_{\alpha_{i}}\subseteq\mathrm{k}\langle AB\rangle$
for any $i=1,\ldots,n$. Hence $\mathrm{k}\langle AB\rangle=\sum_{i=1}%
^{n}V_{\alpha_{i}}$ and
\[
\mathcal{B}\mathrm{k}\langle AB\rangle=\mathcal{B}\sum_{i=1}^{n}V_{\alpha_{i}%
}=\sum_{i=1}^{n}\mathcal{A}_{\alpha_{i}}V_{\alpha_{i}}=\sum_{i=1}^{n}%
V_{\alpha_{i}}=\mathrm{k}\langle AB\rangle.
\]
So $\mathcal{B}\subset\mathcal{H}$. Moreover
\[
\dim_{\mathrm{k}}(AB)+\dim_{\mathrm{k}}(\mathcal{H})\geq\dim_{\mathrm{k}%
}(AB)+\dim_{\mathrm{k}}(\mathcal{B})\geq\dim_{\mathrm{k}}(V_{\alpha_{1}}%
)+\dim_{\mathrm{k}}(\mathcal{A}_{\alpha_{1}})\geq\dim_{\mathrm{k}}%
(A)+\dim_{\mathrm{k}}(B)\label{ineqfund}%
\]
by (\ref{ineqK}) and because $V_{\alpha_{1}}\subset\mathrm{k}\langle
AB\rangle$.

2:
The space $\langle\mathcal{H}A\rangle$ contains $A$ and is finite-dimensional
because both $\mathcal{H}$ and $\mathrm{k}\langle AB\rangle$ are.\ Similarly,
$\langle\mathcal{H}B\rangle$ contains $B$ and is finite-dimensional.\ Let
$A^{\prime}$ and $B^{\prime}$ be finite sets such that $\langle\mathcal{H}%
A\rangle=\langle A^{\prime}\rangle,$ $\langle\mathcal{H}B\rangle=\langle
B^{\prime}\rangle$, $A\subset A^{\prime}$ and $B\subset B^{\prime}$. Observe
first that
\[
\langle A^{\prime}B^{\prime}\rangle=\langle\mathcal{H}A\mathcal{H}%
B\rangle=\langle\mathcal{H}AB\rangle=\langle AB\rangle
\]
for $\mathcal{A}$ is commutative and $\langle\mathcal{H}AB\rangle=\langle
AB\rangle$.\ We then get Assertion 2 applying Assertion 1 to $A^{\prime}$ and
$B^{\prime}$.

\end{proof}

\begin{corollary}
\label{CorAB-1}Let $\mathcal{A}$ be a commutative Banach algebra with no non
trivial finite-dimensional subalgebra. Then for any finite subsets $A$ and $B$
such that $\mathrm{k}\langle A\rangle\cap U(\mathcal{A})\neq\emptyset$ and
$\mathrm{k}\langle B\rangle\cap U(\mathcal{A})\neq\emptyset,$ we have
\[
\dim_{\mathrm{k}}(AB)\geq\dim_{\mathrm{k}}(A)+\dim_{\mathrm{k}}(B)-1.
\]

\end{corollary}

\begin{example}
The previous corollary applies in particular to the Banach algebra
$\mathcal{A}=\mathcal{C}_{0}(I)$ where $I$ is any compact interval in
$\mathbb{R}$ (or more generally $I$ is a compact and connected set) and
$\mathcal{C}_{0}(I)$ is the set of continuous functions $f:I\rightarrow
\mathbb{R}$.\ 
\end{example}

\begin{remark}
Assume $A$ and $B$ as in Theorem \ref{Th_KL_Acom} and $\dim_{\mathrm{k}%
}(A)+\dim_{\mathrm{k}}(B)>\dim_{\mathrm{k}}(\mathcal{A})$. Then, for any
invertible $x\in\mathcal{A}$, we get%
\[
\dim_{\mathrm{k}}(\mathrm{k}\langle AB\rangle\cap\mathcal{H}x)>0\text{.}%
\]
If $\mathcal{H}$ is a field (which is the case when $\mathcal{A}$ is a field),
this shows that $x\in\mathrm{k}\langle AB\rangle$ and we have in this case
$\mathrm{k}\langle AB\rangle=\mathcal{A}$.\ In the general case, $\mathcal{H}$
is not a field and we can have $\mathrm{k}\langle AB\rangle\subsetneqq
\mathcal{A}$. For example, consider $\mathcal{A}=\mathrm{k}^{n}$ with $n\geq3$
and $A=B$ the subalgebra of vectors whose last two coordinates are equal. In
this case, we have $AB=A=B$.
\end{remark}

With Theorem \ref{Th_KL_Acom} in hand, one can prove the following
generalization to arbitrary finite Minkowski products. The proof is similar to
that of Theorem 2.7 in \cite{Lec} so we omit it.

\begin{theorem}
\label{Th_Lin_Kne_n} Assume $\mathcal{A}$ is a commutative finite-dimensional
algebra, a commutative subalgebra of a Banach algebra or a subalgebra of a
product of field extensions over $\mathrm{k}$. Assume $\mathcal{A}$ contains
only a finite number of finite-dimensional subalgebras. Consider a collection
of finite subsets $A_{1},\ldots,A_{n}$ of $\mathcal{A}^{\ast}$ such that
$\mathrm{k}\langle A_{i}\rangle\cap U(\mathcal{A})\neq\emptyset$ for any
$i=1,\ldots,n$. Set $\mathcal{H}:=\mathcal{H}(A_{1}\cdots A_{n})$. The
following statements hold and are equivalent:

\begin{enumerate}
\item $\dim_{\mathrm{k}}(A_{1}\cdots A_{n})\geq\sum_{i=1}^{n}\dim_{\mathrm{k}%
}(A_{i}\mathcal{H})-(n-1)\dim_{\mathrm{k}}(\mathcal{H})$,

\item $\dim_{\mathrm{k}}(A_{1}\cdots A_{n})\geq\sum_{i=1}^{n}\dim_{\mathrm{k}%
}(A_{i})-(n-1)\dim_{\mathrm{k}}(\mathcal{H})$,

\item any one of the above two statements in the case $n=2$.
\end{enumerate}
\end{theorem}

\begin{remark}
In~\cite{Hou}, Hou shows that in the context of non separable extension of
fields, a counterexample to Theorem~\ref{Th_KL_Acom} could only arise with
$\dim A \geq6$. Following the proof given in~\cite{Hou}, we are able to show
that if $\dim A \leq4$, then no hypotheses on $\mathbb{A}(A)$ other than
commutativity is needed to get Theorem~\ref{Th_KL_Acom}.
\end{remark}

\subsection{Remarks on Olson type Theorem}

It is known that Kneser's theorem does not hold for non abelian group. In
\cite{Ols}, Olson gave a weaker version of Kneser's theorem for arbitrary
groups. This Olson theorem admits a natural linearization for division rings
\cite{EL2}.\ It is tempting to look for a possible analogue in our algebras
context.\ In fact, by using the Kemperman linear transform defined in
\cite{EL2} and arguments closed from those we have used to establish Theorem
\ref{Th_KL_Acom} one can prove the following analogue of Olson's theorem where
no hypothesis on the commutativity of $\mathcal{A}$ neither on the number of
its finite-dimensional subalgebras is required.

\begin{theorem}
\label{TH-Ol8lin}Let $\mathcal{A}$ be a unital associative algebra over
$\mathrm{k}$ satisfying \textbf{H}$_{\mathrm{s}}$. Consider $V,W$
finite-dimensional $\mathrm{k}$-vector spaces in $\mathcal{A}$ such that
$V\cap U(\mathcal{A})\neq\emptyset$ and $W\cap U(\mathcal{A})\neq\emptyset$.
Then one of the two following assertions holds

\begin{enumerate}
\item There exists a $\mathrm{k}$-vector subspace $N$ of $\mathrm{k}\langle
VW\rangle$ such that

\begin{itemize}
\item $N\cap U(\mathcal{A})=\emptyset$,

\item $\dim_{\mathrm{k}}\mathrm{k}\langle VW\rangle\geq\dim_{\mathrm{k}}V+
\dim_{\mathrm{k}}W-\dim_{\mathrm{k}}(N)$.
\end{itemize}

\item There exist a $\mathrm{k}$-vector subspace $S$ of $\mathrm{k}\langle
VW\rangle$ and a subalgebra $\mathcal{H}$ of $\mathcal{A}$ such that

\begin{itemize}
\item $S\cap U(\mathcal{A})\neq\emptyset,$

\item $\mathrm{k}\subset\mathcal{H}\subset\mathcal{A},$

\item $\dim_{\mathrm{k}}\mathrm{k}\langle VW\rangle\geq\dim_{\mathrm{k}}%
S\geq\dim_{\mathrm{k}}V+\dim_{\mathrm{k}}W-\dim_{\mathrm{k}}\mathcal{H}$,

\item $\mathcal{H}S=S$ or $S\mathcal{H}=S$.
\end{itemize}
\end{enumerate}
\end{theorem}

Assertion 2 looks indeed as a natural analogue of Olson's theorem for
algebras.\ Also, when $\mathcal{A}$ is a division ring, we must have $N=\{0\}$
in assertion 1.$\ $This is unfortunately not the case in general. Moreover,
for an algebra $\mathcal{A}$, one can have few constraints on the dimensions
of the subspaces $N$ such that $N\cap U(\mathcal{A})=\emptyset$. This is
notably the case of the matrix algebra $M_{n}(\mathbb{C})$ which admits
subspaces $N$ of any dimension less than $n^{2}-n$ containing no invertible
matrices or the Banach algebra of continuous functions from $[0,1]$ to
$\mathbb{R}$. So it eventually appears that this Olson type theorem for
algebras yields only few information on the dimension of the space products
$\mathrm{k}\langle VW\rangle$.

Now, if we assume that $V$ is commutative, we obtain immediately the following
corollary of Lemma \ref{Lem_diff}.

\begin{corollary}
\label{CorOlsonWeak}Assume $\mathcal{A}$ satisfies \textbf{H}$_{\mathrm{s}}%
$\footnote{Observe there is no hypothesis on the number of subalgebras of
$\mathcal{A}$ here.} and consider $V,W$ finite-dimensional $\mathrm{k}$-vector
spaces in $\mathcal{A}$ such that $V$ is commutative, $V\cap U(\mathcal{A}%
)\neq\emptyset$ and $W\cap U(\mathcal{A})\neq\emptyset$. Then there exist a
$\mathrm{k}$-vector subspace $S$ of $\mathrm{k}\langle VW\rangle$ and a
finite-dimensional subalgebra $\mathcal{H}$ of $\mathcal{A}$ such that

\begin{itemize}
\item $S\cap U(\mathcal{A})\neq\emptyset,$

\item $\mathrm{k}\subset\mathcal{H}\subset\mathcal{A},$

\item $\dim_{\mathrm{k}}\mathrm{k}\langle VW\rangle\geq\dim_{\mathrm{k}}%
S\geq\dim_{\mathrm{k}}V+\dim_{\mathrm{k}}W-\dim_{\mathrm{k}}\mathcal{H}$,

\item $\mathcal{H}S=S$.
\end{itemize}
\end{corollary}

\begin{proof}
Choose $a\in V\cap U(\mathcal{A})$. By Lemma \ref{Lem_diff}, there exists a
(commutative) finite-dimensional subalgebra $\mathcal{H}$ of $\mathcal{A}$
such that $\mathrm{k}\subseteq\mathcal{H}\subseteq\mathbb{A}(A)$ and a vector
space $S$ contained in $\mathrm{k}\langle VW\rangle$ such that $S\cap
U(\mathcal{A})\neq\emptyset$, $\mathcal{H}S=S$ and
\[
\dim_{\mathrm{k}}(S)+\dim_{\mathrm{k}}(\mathcal{H})\geq\dim_{\mathrm{k}%
}(V)+\dim_{\mathrm{k}}(W).
\]

\end{proof}

\section{Hamidoune and Tao type results}

In this section, we assume $\mathcal{A}$ satisfies \textbf{H}$_{\mathrm{s}}$
so that every subspace of $\mathcal{A}$ containing an invertible element
admits a basis of invertible elements (Proposition~\ref{Prop_base_inv}).

\subsection{Linear hamidoune's connectivity}

The notion of connectivity for a subset $S$ of a group $G$ was developed by
Hamidoune in \cite{Hami2}.\ As suggested by Tao in \cite{Tao}, it is
interesting to generalize Hamidoune's definition by introducing an additional
parameter $\lambda$. The purpose of this paragraph is to adapt this notion of
connectivity to our algebra context. Assume $V$ is a finite-dimensional fixed
$\mathrm{k}$-subspace of $\mathcal{A}$ such that $V\cap U(\mathcal{A}%
)\neq\emptyset$ and $\lambda$ is a real parameter. For any finite-dimensional
$\mathrm{k}$-subspace $W$ of $\mathcal{A}$, we define%
\begin{equation}
c(W):=\dim_{\mathrm{k}}(\mathrm{k}\langle WV \rangle)-\lambda\dim_{\mathrm{k}%
}(W). \label{connectivity}%
\end{equation}
For any $x\in U(\mathcal{A})$, we have immediately that $c(xW)=c(W).\;$

\begin{lemma}
For any finite-dimensional subspaces $W_{1},W_{2}$ and $V$ of $\mathcal{A}$,
we have%
\[
c(W_{1}+W_{2})+c(W_{1}\cap W_{2})\leq c(W_{1})+c(W_{2}).
\]

\end{lemma}

\begin{proof}
We have
\begin{equation}
\dim_{\mathrm{k}}(W_{1}+W_{2})+\dim_{\mathrm{k}}(W_{1}\cap W_{2}%
)=\dim_{\mathrm{k}}(W_{1})+\dim_{\mathrm{k}}(W_{2}) \label{dimform}%
\end{equation}
and%
\[
\dim_{\mathrm{k}}(\mathrm{k}\langle W_{1}V\rangle+\mathrm{k}\langle
W_{2}V\rangle)+\dim_{\mathrm{k}}(\mathrm{k}\langle W_{1}V\rangle\cap
\mathrm{k}\langle W_{2}V\rangle)=\dim_{\mathrm{k}}(W_{1}V)+\dim_{\mathrm{k}%
}(W_{2}V).
\]
Observe that $\mathrm{k}\langle(W_{1}+W_{2})\cdot V\rangle=\mathrm{k}\langle
W_{1}V\rangle+\mathrm{k}\langle W_{2}V\rangle$ and $\mathrm{k}\langle
(W_{1}\cap W_{2})\cdot V\rangle\subseteq\mathrm{k}\langle W_{1}V\rangle
\cap\mathrm{k}\langle W_{2}V\rangle$.\ This gives%
\begin{equation}
\dim_{\mathrm{k}}(\mathrm{k}\langle(W_{1}+W_{2})\cdot V\rangle)+\dim
_{\mathrm{k}}(\mathrm{k}\langle W_{1}\cap W_{2})\cdot V \rangle)\leq
\dim_{\mathrm{k}}(\mathrm{k}\langle W_{1}V\rangle)+ \dim_{\mathrm{k}%
}(\mathrm{k}\langle W_{2}V\rangle). \label{ineq}%
\end{equation}
We then obtain the desired equality by subtracting from (\ref{ineq}),
$\lambda$ copies of (\ref{dimform}).
\end{proof}

\bigskip

Similarly to \cite{Hami2}, we define the \emph{connectivity} $\kappa
=\kappa(V)$ as the infimum of $c(W)$ over all finite-dimensional $\mathrm{k}%
$-subspaces of $\mathcal{A}$ such that $W\cap U(\mathcal{A})\neq\emptyset$.\ A
\emph{fragment} of $V$ is a finite-dimensional $\mathrm{k}$-subspace of
$\mathcal{A}$ which attains the infimum $\kappa$. An \emph{atom} of $V$ is a
fragment of minimal dimension. Since $c(xW)=c(W)$ for any $x\in U(\mathcal{A}%
)$, any left translate by an invertible of a fragment is a fragment and any
left translate of an atom is an atom. Since $\dim_{\mathrm{k}}(WV)\geq
\dim_{\mathrm{k}}(W)$ (because $V\cap U(\mathcal{A})\neq\emptyset$), we have%
\begin{equation}
c(W)\geq(1-\lambda)\dim_{\mathrm{k}}(W). \label{maj_connecti}%
\end{equation}
We observe that when $\lambda<1$, $c(W)$ is always positive and takes a
discrete set of values.\ Therefore, when $\lambda\leq1$, there exists at least
one fragment and at least one atom. \emph{In the remainder of this paragraph
we will assume that }$0<\lambda\leq1$. Let $W_{1}$ and $W_{2}$ be two
fragments such that $W_{1}\cap W_{2}$ intersects $U(\mathcal{A})$. By the
previous lemma, we derive%
\[
c(W_{1}+W_{2})+c(W_{1}\cap W_{2})\leq c(W_{1})+c(W_{2})=2\kappa.
\]
Since $W_{1}+W_{2}$ and $W_{1}\cap W_{2}$ are finite-dimensional and
intersects $U(\mathcal{A})$, we must have $c(W_{1}+W_{2})\geq\kappa$ and
$c(W_{1}\cap W_{2})\geq\kappa$.\ Hence $c(W_{1}+W_{2})=c(W_{1}\cap
W_{2})=\kappa$. This means that $W_{1}+W_{2}$ and $W_{1}\cap W_{2}$ are also
fragments. If we assume now that $W_{1}$ and $W_{2}$ are atoms such that
$W_{1}\cap W_{2}$ intersects $U(\mathcal{A})$, we obtain that $W_{1}=W_{2}$.

\begin{proposition}
\label{Prop_connect}\label{Cor_hami} Assume $\mathcal{A}$ satisfies
\textbf{H}$_{\mathrm{s}}$ and $V$ is a finite-dimensional fixed $\mathrm{k}%
$-subspace of $\mathcal{A}$ such that $V\cap U(\mathcal{A})\neq\emptyset$

\begin{enumerate}
\item There exists a unique atom $\mathcal{H}_{\lambda}$ containing $1$ for
$V$.

\item This atom is a subalgebra of $\mathcal{A}$ containing $\mathcal{H}%
$\textbf{$_{l}(V)$}.

\item Moreover the atoms of $V$ which intersect $U(\mathcal{A})$ are the right
$\mathcal{H}_{\lambda}$-modules $x\mathcal{H}_{\lambda}$ where $x$ runs over
$U(\mathcal{A})$.

\item For any finite-dimensional $\mathrm{k}$-subspace $W$ satisfying $W\cap
U(\mathcal{A})\neq\emptyset$, we have
\[
\dim_{\mathrm{k}}(\mathrm{k}\langle WV \rangle)\geq\lambda\dim_{\mathrm{k}%
}(W)+\dim_{\mathrm{k}}(V)-\lambda\dim_{\mathrm{k}}(\mathcal{H}_{\lambda})
\]

\end{enumerate}
\end{proposition}

\begin{proof}
Since there exists at least one atom and the left translate of any atom by any
invertible is an atom, there exists one atom $\mathcal{H}$ containing $1$.
Now, this atom must be unique. Indeed, if $\mathcal{H}^{\prime}$ is another
atom containing $1,$ we have that $\mathcal{H}\cap\mathcal{H}^{\prime}$
intersects $U(\mathcal{A})$. Hence, by the previous arguments $\mathcal{H}%
=\mathcal{H}^{\prime}$. Now, for any $h\in\mathcal{H}\cap U(\mathcal{A})$, we
have that $\mathcal{H}\cap h^{-1}\mathcal{H}$ contains $1$.\ Since both
$\mathcal{H}$ and $h^{-1}\mathcal{H}$ are atoms, we must have $h^{-1}%
\mathcal{H}=\mathcal{H}$ and $\mathcal{H}=h\mathcal{H}$. So $\mathcal{H}$ is
stable under multiplication by any invertible of $\mathcal{H}$. By Proposition
\ref{Prop_base_inv}, $\mathcal{H}$ is then stable by multiplication.\ We then
deduce that $\mathcal{H}$ is a subalgebra of $\mathcal{A}$ by
Lemma~\ref{lem_stab}. Moreover, if $x\in\mathcal{H}_{l}(V)$ and\textbf{
}$x\notin\mathcal{H}$, then $(\mathcal{H}+kx)V=\mathcal{H}V$ and so
$c(\mathcal{H}+kx)<c(\mathcal{H})$ since $\lambda>0$ contradicting the
definition of an atom. Finally, given any atom $W$ of $V$ intersecting
$U(\mathcal{A})$, we must have $w^{-1}W=\mathcal{H}$ for any $w\in W\cap
U(\mathcal{A})$ since $\mathcal{H}$ is the unique atom containing $1$ and
$w^{-1}\mathcal{H}$ is an atom containing $1$.

Let us now prove~(4). By definition of $\kappa$, we have $\kappa
=c(\mathcal{H}_{\lambda})\leq c(W)$.\ This gives%
\[
\dim_{\mathrm{k}}(\mathrm{k}\langle\mathcal{H}_{\lambda}V\rangle)-\lambda
\dim_{\mathrm{k}}(\mathcal{H}_{\lambda})\leq\dim_{\mathrm{k}}(\mathrm{k}%
\langle WV\rangle)-\lambda\dim_{\mathrm{k}}(W).
\]
We thus get
\[
\dim_{\mathrm{k}}(\mathrm{k}\langle WV\rangle)\geq\lambda\dim_{\mathrm{k}%
}(W)+\dim_{\mathrm{k}}(\mathcal{H}_{\lambda}V)-\lambda\dim_{\mathrm{k}%
}(\mathcal{H}_{\lambda})\geq\lambda\dim_{\mathrm{k}}(W)+\dim_{\mathrm{k}%
}(V)-\lambda\dim_{\mathrm{k}}(\mathcal{H}_{\lambda}).
\]

\end{proof}


\begin{remark}
\ 

\begin{enumerate}
\item Assume $\mathcal{A}$ has no non trivial f.d. subalgebra. Then we must
have $\mathcal{H}_{\lambda}=\mathrm{k}$ for any $\lambda\leq1$. So we obtain
\[
\dim_{\mathrm{k}}(\mathrm{k}\langle WV\rangle)\geq\lambda(\dim_{\mathrm{k}%
}(W)-1)+\dim_{\mathrm{k}}(V)\text{ for any }\lambda\leq1\text{.}%
\]
In particular, for $\lambda=1$, this gives%
\[
\dim_{\mathrm{k}}(\mathrm{k}\langle WV\rangle)\geq\dim_{\mathrm{k}}%
(W)+\dim_{\mathrm{k}}(V)-1
\]
which generalizes Corollary \ref{CorAB-1}.

\item If $V$ and $W$ are such that $\dim_{\mathrm{k}}(\mathrm{k}\langle WV
\rangle)<\dim_{\mathrm{k}}(W)+\dim_{\mathrm{k}}(V)-1$, then the unique atom
$\mathcal{H}_{1}$ for $V$ containing $1$ is a subalgebra of $\mathcal{A}$ of
dimension at least $2$.

\item Contrary to Theorem \ref{Th_KL_Acom} where the lower bound makes appear
the stabilizer of $\mathrm{k}\langle WV\rangle$, the subalgebra $\mathcal{H}%
_{\lambda}$ in the previous corollary only depends on $\lambda$ and $V$ and is
the same for each subspace $W$.
\end{enumerate}
\end{remark}

\subsection{Tao's theorem for algebras}

\label{subsec-linTao}We say that $V=\mathrm{k}\langle A\rangle$, where $A$ is
a finite subset of $\mathcal{A}$, is a \emph{space of small doubling}, when
$\dim_{\mathrm{k}}(A^{2})=O(\dim_{\mathrm{k}}(A)).\;$Simplest examples of
spaces of small doubling are the spaces $V=\mathrm{k}\langle A\rangle$
containing $1$ and such that $\dim_{\mathrm{k}}(A^{2})=\dim_{\mathrm{k}}(A)$.
Then by Lemma \ref{lem_stab}, $V$ is a subalgebra containing $\mathrm{k}$. In
general, a space of small doubling $\mathrm{k}\langle A\rangle$ is not a
subalgebra and neither a left nor right $\mathcal{H}$-module for a subalgebra
$\mathrm{k}\subseteq\mathcal{H}\subseteq\mathcal{A}$. The following theorem,
which is a linear version of Theorem 1.2 in \cite{Tao}, permits to study the
spaces of small doubling in an algebra $\mathcal{A}$ satisfying \textbf{H}%
$_{\mathrm{s}}$.

\begin{theorem}
\label{Th_linTao}Consider finite-dimensional $\mathrm{k}$-subspaces $V$ and
$W$ of $\mathcal{A}$ (satisfying \textbf{H}$_{\mathrm{s}}$) intersecting
$U(\mathcal{A})$ such that $\dim_{\mathrm{k}}(W)\geq\dim_{\mathrm{k}}(V)$ and
$\dim_{\mathrm{k}}(\mathrm{k}\langle WV\rangle)\leq(2-\varepsilon
)\dim_{\mathrm{k}}(V)$ for some real $\varepsilon$ such that $0<\varepsilon
<2$. Then, there exists a finite-dimensional subalgebra $\mathcal{H}$ such
that $\dim_{\mathrm{k}}(\mathcal{H})\leq(\frac{2}{\varepsilon}-1)\dim
_{\mathrm{k}}(V)$, and $V$ is contained in the left $\mathcal{H}$-module
$\mathcal{H}V$ with $\dim_{\mathrm{k}}(\mathcal{H}V)\leq(\frac{2}{\varepsilon
}-1)\dim_{\mathrm{k}}(\mathcal{H})$.
\end{theorem}

\begin{proof}
We apply linear Hamidoune connectivity with $\lambda=1-\frac{\varepsilon}{2}$.
We have by (\ref{maj_connecti}) $c(S)\geq\frac{\varepsilon}{2}\dim
_{\mathrm{k}}(S)$ for any $\mathrm{k}$-subspace $S$.\ This can be rewritten as%
\begin{equation}
\dim_{\mathrm{k}}(S)\leq\frac{2}{\varepsilon}c(S). \label{maj3}%
\end{equation}
We also get%
\[
c(W):=\dim_{\mathrm{k}}(\mathrm{k}\langle WV \rangle)-(1-\frac{\varepsilon}%
{2})\dim_{\mathrm{k}}(W)\leq(2-\varepsilon)\dim_{\mathrm{k}}(V)-(1-\frac
{\varepsilon}{2})\dim_{\mathrm{k}}(V)=(1-\frac{\varepsilon}{2})\dim
_{\mathrm{k}}(V).
\]
since $\dim_{\mathrm{k}}(WV)\leq(2-\varepsilon)\dim_{\mathrm{k}}(V)$ and
$\dim_{\mathrm{k}}(W)\geq\dim_{\mathrm{k}}(V)$. By Proposition
\ref{Prop_connect}, the unique atom containing $1$ is a subalgebra
$\mathcal{H}$.\ By definition of an atom, we should have%
\[
\kappa=c(\mathcal{H})\leq c(W)\leq(1-\frac{\varepsilon}{2})\dim_{\mathrm{k}%
}(V).
\]
We therefore obtain, by using (\ref{maj3}) with $S=\mathcal{H}$, that%
\[
\dim_{\mathrm{k}}(\mathcal{H})\leq\frac{2}{\varepsilon}c(\mathcal{H})\leq
\frac{2}{\varepsilon}c(W)\leq(\frac{2}{\varepsilon}-1)\dim_{\mathrm{k}}(V).
\]
By using that $c(\mathcal{H})=\dim_{\mathrm{k}}(\mathcal{H}V)-(1-\frac
{\varepsilon}{2})\dim_{\mathrm{k}}(\mathcal{H})$ and the previous inequality
$c(\mathcal{H})\leq(1-\frac{\varepsilon}{2})\dim_{\mathrm{k}}(V)$, we get%
\begin{equation}
\dim_{\mathrm{k}}(\mathcal{H}V)\leq(1-\frac{\varepsilon}{2})\dim_{\mathrm{k}%
}(V)+(1-\frac{\varepsilon}{2})\dim_{\mathrm{k}}(\mathcal{H}). \label{ineq3}%
\end{equation}
We can also bound $\dim_{\mathrm{k}}(V)$ by $\dim_{\mathrm{k}}(\mathcal{H}V)$
in (\ref{ineq3}). This yields
\begin{equation}
\dim_{\mathrm{k}}(\mathcal{H}V)\leq(\frac{2}{\varepsilon}-1)\dim_{\mathrm{k}%
}(\mathcal{H})\text{.} \label{ineq4}%
\end{equation}

\end{proof}

\pagebreak

\begin{remark}
\ 

\begin{enumerate}
\item When $\dim_{\mathrm{k}}(V^{2})\leq(2-\varepsilon)\dim_{\mathrm{k}}(V)$,
we can apply Theorem \ref{Th_linTao} with $V=W$ and obtain that $V\subset
\mathcal{H}V$ with
\[
\dim_{\mathrm{k}}(\mathcal{H})\leq(\frac{2}{\varepsilon}-1)\dim_{\mathrm{k}%
}(V)\quad\text{ and }\quad\dim_{\mathrm{k}}(\mathcal{H}V)\leq(\frac
{2}{\varepsilon}-1)\dim_{\mathrm{k}}(\mathcal{H}).
\]

\item When $\mathcal{A}$ has no non trivial f.d.\negthinspace\negthinspace
\ subalgebra and $\dim_{\mathrm{k}}(V^{2})\leq(2-\varepsilon)\dim_{\mathrm{k}%
}(V)$, we have $\mathcal{H}=\mathrm{k}$. So
\[
\frac{1}{\frac{2}{\varepsilon}-1}\leq\dim_{\mathrm{k}}(V)\leq\frac
{2}{\varepsilon}-1.
\]

\end{enumerate}
\end{remark}

\section{Groups setting}

\subsection{Recovering results in the group setting}

The aim of this paragraph is to explain how Theorem~\ref{Th_KL_Acom} permits
to recover Diderrich's theorem for groups.
The proof goes through three steps. First we turn the group $\mathbb{G}$ into
the group algebra $\mathrm{k}[\mathbb{G}]$. Second, we link the stabilizer in
$\mathbb{G}$ of the subset $A$ with the stabilizer in $\mathrm{k}[\mathbb{G}]$
of the subspace $\mathrm{k}\langle A\rangle$. Third we choose a convenient
field (the field $\mathbb{C}$ of complex numbers) so that the subalgebra
generated by $\mathrm{k}\langle A\rangle$ has only a finite number of subalgebras.

Let us now detail these ideas. First, the group algebra $\mathrm{k}%
[\mathbb{G}]$ is the $\mathrm{k}$-vector space with basis $\{e_{g}\mid
g\in\mathbb{G}\}$ and multiplication defined by $e_{g}\cdot e_{g^{\prime}%
}=e_{gg^{\prime}}$ for any $g,g^{\prime}$ in $\mathbb{G}$.\
Given any nonempty set $A$ in $\mathbb{G}$, we define its associated set in
$\mathrm{k}[\mathbb{G}]$ as $\overline{A}=\{e_{a}\mid a\in A\}$. It is clear
that $A$ is a commutative set in $\mathbb{G}$ if and only if $\overline{A}$ is
a commutative set in $\mathrm{k}[\mathbb{G}]$.\ In that case, the subalgebra
$\mathbb{A}(\overline{A})$ is a finite-dimensional commutative algebra
isomorphic to $\mathrm{k}[\mathbb{G}(A)]$ the group algebra of the subgroup
$\mathbb{G}(A)$ of $\mathbb{G}$ generated by the elements of $A$. Moreover,
write%
\[
H=\{h\in\mathbb{G}\mid hA=A\}\quad\text{ and }\quad\mathcal{H}_{l}%
=\{x\in\mathrm{k}[\mathbb{G}]\mid x\overline{A}\subset\mathrm{k}%
\langle\overline{A}\rangle\}
\]
for the left stabilizer of $A$ in $\mathbb{G}$ and the left stabilizer of
$\mathrm{k}\langle\overline{A}\rangle$ in $\mathrm{k}[\mathbb{G}]$, respectively.

\begin{lemma}
\label{Lem_stab}We have $\mathcal{H}_{l}=\mathrm{k}\langle\overline{H}%
\rangle=\mathrm{k}[H]$ that is, $\mathcal{H}_{l}$ is the group algebra of the
group $H$.
\end{lemma}

\begin{proof}
The inclusion $\mathcal{H}_{l}\supset\mathrm{k}\langle\overline{H}\rangle$ is
immediate. For the converse, observe first that for any $g\notin H$, there
exists $a_{g}$ in $A$ such that $ga_{g}\notin A$.\ Consider $x=\sum_{g\in
G}\lambda_{g}e_{g}$ in $\mathcal{H}_{l}$ (where the coefficients $\lambda_{g}$
are all but a finite number equal to zero when $\mathbb{G}$ is infinite).
Since $\mathcal{H}_{l}\supset\mathrm{k}\langle\overline{H}\rangle$, we may
assume that $\lambda_{g}=0$ for all~$g \in H$ and write $x=\sum_{g\notin
H}\lambda_{g}e_{g}$. Our aim is to show that $\lambda_{g}=0$ for all~$g \notin
H$. For such a $g$, there exists $a \in A$ such that $ga \notin A$. Moreover,
since $x \in\mathcal{H}_{l}$ and $a \in A$, we have $xe_{a} \in\mathrm{k}%
[\overline{A}]$. Finally, we get%

\[
\sum_{g^{\prime}\in A}\mu_{g^{\prime}}e_{g^{\prime}}=xe_{a}=\sum_{g^{\prime
}\notin H}\lambda_{g^{\prime}}e_{g^{\prime}}e_{a}=\sum_{g^{\prime}\notin
H}\lambda_{g^{\prime}}e_{g^{\prime}a}%
\]
Since the family $\{e_{g^{\prime}},g^{\prime}\in\mathbb{G}\}$ is a basis for
$\mathrm{k}[\mathbb{G}]$, we get $\lambda_{g}=0$ comparing the coefficient of
$e_{ga}$ in the left and right hand side of the previous equality.
\end{proof}

\bigskip

To obtain Diderrich's theorem for groups from Theorem~\ref{Th_KL_Acom} and
Lemma \ref{Lem_stab}, we have to find a field such that $\mathbb{A}%
(\overline{A})=\mathrm{k}[\mathbb{G}(A)]$ admits a finite number of
subalgebras containing $1$ and $\mathrm{k}[\mathbb{G}]$ verifies~\textbf{H}%
$_{\mathrm{s}}$. These two points relies on the two following lemmas.


\begin{lemma}
\label{Lem_alg_gr} Let $\mathbb{G}$ be a finitely generated commutative group.
Then
\[
\mathbb{C}[\mathbb{G}]\mathbb{\simeq C}[{X_{1}}^{\pm1},\ldots,{X_{r}}^{\pm
1}]^{m}%
\]
thus is a product of integral algebras and has a finite number of finite
dimensional subalgebras.
\end{lemma}

\begin{proof}
This is standard representation theory. We may write $\mathbb{G}%
\simeq\mathbb{Z}^{r}\times\mathbb{G}^{\prime}$ with $\mathbb{G}^{\prime}$ a finite
group of order $m$. We then have $\mathbb{C}[\mathbb{G}]\simeq\mathbb{C}%
[\mathbb{Z}^{r}]\otimes\mathbb{C}[\mathbb{G^{\prime}}]$.

The algebra $\mathbb{C}[\mathbb{G^{\prime}}]$ is semisimple and then
isomorphic to a product of $\ell$ matrix algebras where $\ell$ is the number
of conjugacy classes in $\mathbb{G^{\prime}}$. So $\ell=|\mathbb{G^{\prime}%
}|=m$ since $\mathbb{G^{\prime}}$ is commutative. A dimension argument (or
commutation argument) show that all the matrix algebras have to be of
dimension~$1$. Finally $\mathbb{C}[\mathbb{G^{\prime}}] \simeq\mathbb{C}^{m}$.

Moreover, we also have $\mathbb{C}[\mathbb{Z}^{r}]\mathbb{\simeq C}[{X_{1}%
}^{\pm1},\ldots,{X_{r}}^{\pm1}]$ whose unique finite-dimensional subalgebra is
$\mathbb{C}$.

Finally, we obtain $\mathbb{C}[\mathbb{G}]\simeq\mathbb{C}[{X_{1}}^{\pm
1},\ldots,{X_{r}}^{\pm1}]^{m}$. This implies that the finite-dimensional
subalgebras of $\mathbb{C}[\mathbb{G}]$ are the subalgebra of $\mathbb{C}^{m}%
$.\ There thus exists only finitely many such subalgebras.
\end{proof}

\begin{lemma}
\label{Lem_alg_gr_banach} Let $\mathbb{G}$ be a finitely generated group. Then
$\mathbb{C}[\mathbb{G}]$ can be identified with a subalgebra of a Banach
algebra over $\mathbb{C}$.
\end{lemma}

\begin{proof}
Let us consider on $\mathbb{C}[\mathbb{G}]$ the norm defined by
\[
\|\sum_{g \in\mathbb{G}} \lambda_{g} e_{g} \|= \sum_{g \in\mathbb{G}}
|\lambda_{g}|\,.
\]
For $x,y \in\mathbb{C}[\mathbb{G}]$, we have $\|xy\| \leq\|x \| \|y \|$. The
completion of $\mathbb{C}[\mathbb{G}]$ will then be a Banach algebra.
\end{proof}

\begin{corollary}
[Diderrich's theorem for groups]Consider $A$ and $B$ be two finite nonempty
subsets of a group $\mathbb{G}$. Assume that $A$ is commutative. Let
$H:=\{g\in G\mid gAB=AB\}.$ Then%
\[
\left\vert AB\right\vert \geq\left\vert A\right\vert +\left\vert B\right\vert
-\left\vert H\right\vert .
\]

\end{corollary}

\begin{proof}
Since $AB$ belongs to the subgroup of $\mathbb{G}$ generated by the finite
sets $A$ and $B$, we can assume that $G$ is finitely generated.
Lemma~\ref{Lem_alg_gr_banach} then shows that $\mathbb{C}[\mathbb{G}]$
satisfies \textbf{H}$_{\mathrm{w}}$. We then apply Theorem \ref{Th_KL_Acom} to
$\overline{A}$ and $\overline{B}$ which consists of invertible elements in
$\mathbb{C}[\mathbb{G}]$ .\ We have $\left\vert A\right\vert =\dim
_{\mathbb{C}}\overline{A},$ $\left\vert B\right\vert =\dim_{\mathbb{C}%
}\overline{B}$ and by Lemma \ref{Lem_stab}, we have $\left\vert H\right\vert
=\dim_{\mathbb{C}}\overline{H}=\dim_{\mathbb{C}}\mathcal{H}$ where
$\mathcal{H}=\{x\in\mathbb{C}[G]\mid x\langle\overline{A}\,\overline{B}%
\rangle=\langle\overline{A}\,\overline{B}\rangle\}$. Since $\mathbb{C}%
[\mathbb{G}(A)]=\mathbb{A}[\overline{A}]$ admits a finite number of
finite-dimensional subalgebras by Lemma \ref{Lem_alg_gr}, we are done.
\end{proof}

\bigskip


\begin{remark}
Observe that in the case of a commutative group $\mathbb{G}$,
Lemma~\ref{Lem_alg_gr_banach} is not necessary. Indeed, we may consider the
commutative finitely generated group $\langle A \cup B \rangle$ whose group
algebra verifies~\textbf{H}$_{\mathrm{w}}$ by Lemma~\ref{Lem_alg_gr}.
\end{remark}



\medskip

\begin{remark}
Contrary to \cite{Xian}, we recover here the results for the groups without
using any Galois correspondence arguments which would become problematic in
the noncommutative case.
\end{remark}

\section{Monoid setting}

Let $M$ be a multiplicative monoid with neutral element $1$. Its set of
invertible elements is defined as%
\[
U(M)=\{x\in M\mid\exists y\in M,\ \ xy=yx=1\}\text{.}%
\]
We denote by $\mathbb{C}[M]$ its monoid algebra over $\mathbb{C}$. Given a
nonempty set $A$ in $G$, we define $\overline{A}=\{e_{a}\mid a\in A\}$ as in
the case of a group algebra. Moreover, we also write%
\[
H_A=\{h\in M\mid hA=A\}\quad\text{ and }\quad\mathcal{H}_{l}(A)=\{x\in
\mathrm{k}[M]\mid x\overline{A}\subset\mathrm{k}\langle\overline{A}\rangle\}
\]
for the left stabilizer of $A$ in $M$ and the left stabilizer of
$\mathrm{k}\langle\overline{A}\rangle$ in $\mathrm{k}[M]$, respectively. It is
clear that $H_A$ is a submonoid of $M$ and $\mathcal{H}_{l}(A)$ a subalgebra of
$\mathrm{k}[M]$. Nevertheless, Lemma \ref{Lem_stab} does not hold in general
when $M$ is not a group as illustrated by the following example.

\begin{example}
Consider $M$ defined as the quotient of the free monoid $\{a,b\}^{\ast}$ (with
neutral element the empty word) by the relations%
\[
a^{2}=b^{2}=ab=ba.
\]
Given $x\in M$, let $\ell(x)$ be the common length of the words of $x$
regarded as a class in $\{a,b\}^{\ast}$. Then $\ell(xy)=\ell(x)+\ell(y)$ for
any $x,y$ in $M$. For $A=\{1,a,b\}$, we thus have $H_A=\{1\}$. Nevertheless, the
subalgebra $\mathcal{H}_{l}(A)$ is not reduced to $\mathbb{C}$. One easily
verifies that it coincides with the $2$-dimensional subalgebra $\mathcal{H}%
_{l}(A)=\mathbb{C}\oplus\mathbb{C}x$ generated by $x=a-b$ with $x^{2}=0$.
\end{example}

\subsection{Finite monoids}

The Kneser theorem for abelian groups becomes false in commutative finite
monoids even if we assume the subsets considered intersect non trivially the
set of invertible elements. To see this, define a monoid $M$ as the quotient
of the free monoid $\{a,b\}^{\ast}$ (with neutral element the empty word) by
the relations%
\begin{equation}
a^{2}=b^{2}=ab=ba\text{ and }a^{4}=a.\label{finiteM}%
\end{equation}
Then $M=\{1,a,b,a^{2},a^{3}\}$ is finite.\ For $A=B=\{1,a,b\}$, we have yet
$A^{2}=\{1,a,b,a^{2}\}$ and $H_{AB}=\{1\}$ whereas%
\[
4=\left\vert A^{2}\right\vert \ngeqslant2\left\vert A\right\vert -\left\vert
H_{AB}\right\vert =5.
\]
It is nevertheless possible to obtain a Hamidoune type theorem from our
algebra setting.

\begin{theorem}
Let $M$ be a finite monoid and $A$ a finite subset in $M$ satisfying $A\cap
U(M)\neq\emptyset$. Then, for any $0<\lambda\leq1$, the subalgebra
$\mathcal{H}_{\lambda}$ of $\mathbb{C}[M]$ which is the unique atom containing
$1$ contains $\mathcal{H}_{l}(A)$ and verifies
\[
\left\vert BA\right\vert \geq\lambda\left\vert A\right\vert +\left\vert
B\right\vert -\lambda\dim_{\mathbb{C}}(\mathcal{H}_{\lambda})\quad\text{ and
}\quad\dim_{\mathbb{C}}(\mathcal{H}_{\lambda})\geq\left\vert H_A\right\vert
\]
for any finite subset $B$ in $M$ such that $B\cap U(M)\neq\emptyset$. In
particular $\mathcal{H}_{1}$ verifies%
\[
\left\vert BA\right\vert \geq\left\vert A\right\vert +\left\vert B\right\vert
-\dim_{\mathbb{C}}(\mathcal{H}_{1})\quad\text{ and }\quad\dim_{\mathbb{C}%
}(\mathcal{H}_{1})\geq\left\vert H_A\right\vert \text{.}%
\]
\end{theorem}

\begin{proof}
Since $M$ is finite, $\mathbb{C}[M]$ verifies \textbf{H}$_{\mathrm{s}}$. We
apply Corollary \ref{Cor_hami} to $\overline{A}$ and $\overline{B}$. We have
$\dim_{\mathbb{C}}(\mathcal{H}_{\lambda})\geq\left\vert H\right\vert $ because
$\mathbb{C}[H]\subset\mathcal{H}_{l}(A)\subset\mathcal{H}_{\lambda}$.
\end{proof}

\begin{remark}
\ 

\begin{enumerate}
\item Contrary to the group setting, for $A$ a finite subset in $M$, the
subalgebra $\mathbb{A}[\overline{A}]$ of $\mathbb{C}[M]$ generated by
$\overline{A}$ can admit an infinite number of finite-dimensional subalgebras.
This is notably the case when $A$ contains a (non invertible) element $a$
generating a submonoid
\[
\langle a\rangle=\{1,a,\ldots,a^{m},a^{m+1},\ldots,a^{m+r-1}\}
\]
with $a^{m+r}=a^{m}$. Then $\mathbb{A}[\overline{A}]$ admits a subalgebra
isomorphic to $\mathbb{C}[X]/(X^{m+r}-X^{m})$. We get
\[
\mathbb{C}[X]/(X^{m+r}-X^{m})\simeq\mathbb{C}[X]/(X^{r}-1)\times
\mathbb{C}[X]/(X^{m})
\]
thus $\mathbb{C}[\overline{A}]$ has an infinite number of finite-dimensional
subalgebras as soon $m>3$ by Theorem \ref{TH_classification}. So we cannot
state a general monoid version of the Diderrich-Kneser theorem from Theorem
\ref{Th_Lin_Kne_n}.

\item But, we have the following version : let $M$ be a monoid whose elements are right regular and 
$A$ be a finite commutative subset of $M$ and $B$ a finite subset of $M$, then 
$|AB| + |H_{AB}|\geq |A|+|B|$. Indeed, $M$ may not be a submonoid of a group, but $A$ does since it is commutative. 
So $\mathbb{C}[\mathbb{A}(A)]$ is a subalgebra of a finitely generated commutative group and thus 
has only a finite number of subalgebra by Lemma~\ref{Lem_alg_gr}. Moreover, adapting the 
proof of Lemma~\ref{Lem_alg_gr_banach}, we get that $\mathbb{C}[M]$ is a subalgebra of a Banach algebra.
Finally, in this context Lemma~\ref{Lem_stab} still holds, since $e_{g'a} \neq e_{g''a}$ if $g' \neq g''$
(using notation of Lemma~\ref{Lem_stab}). Therefore, we can apply Theorem~\ref{Th_KL_Acom}.

\item It is also possible to get from the monoid algebra
$\mathbb{C}[M]$ monoid versions of Corollary \ref{CorOlsonWeak} and Theorem
\ref{Th_linTao}.\ They are left to the reader. Here also we have to use
finite-dimensional subalgebras of $\mathbb{C}[M]$ instead of submonoids of $M$.
\end{enumerate}
\end{remark}

\subsection{Finitely generated commutative monoids}

Let $M$ be a finitely generated commutative monoid.\ Its monoid algebra
$\mathbb{C}[M]$ is a finitely generated algebra (over $\mathbb{C}$). It thus
can be written as $\mathbb{C}[X_{1},\ldots,X_{r}]/I$ where $I$ is an ideal of
$\mathbb{C}[X_{1},\ldots,X_{r}]$. We now give a sufficient condition on the
components of the algebraic variety $V$ defined by $I$ to apply Theorem
\ref{Th_KL_Acom}.

Assume that $I$ is a radical ideal (or that $\mathbb{C}[M]$ is reduced) and
that the irreducible components of $V$ coincide with its connected components.
Then in this case $\mathbb{C}[M]$ is a finite product of integral algebras,
one for each irreducible component of $V$. Thus $\mathbb{C}[M]$ satisfies
hypothesis \textbf{H$_{\mathrm{w}}$} and we can apply Theorem \ref{Th_KL_Acom}
because we know by Lemma \ref{Lem_alg_gr} that it admits a finite number of
finite dimensional subalgebras.

To prove that $\mathbb{C}[M]$ is a finite product of integral algebras, write
$V=V_{1}\cup\cdots\cup V_{s}$ where the $V_{i}$ are the irreducible components
of $V$. For a subset $X$ of $\mathbb{C}^{r}$, write $I(X)$ for the ideal of
$\mathbb{C}[X_{1},\ldots,X_{r}]$ of polynomial vanishing on all $x\in X$. In
particular, we have by using the Nullstellensatz that $I=I(V)$ since
$\mathbb{C}[M]$ is reduced. We also get $I=I(V)=I(V_{1})\cap\cdots\cap
I(V_{s})$. Moreover, since the $V_{i}$ are the connected component of $V$, we
have $V_{i}\cap V_{j}=\emptyset$ for $i\neq j$. The Nullstellensatz also
ensures us that $I(V_{i})+I(V_{j})=\mathbb{C}[X_{1},\ldots,X_{r}]$
(see~\cite{Fulton}, Chapter 1 p.20). The Chinese remainder theorem allows us
to write
\[
\mathbb{C}[M]=\mathbb{C}[X_{1},\ldots,X_{r}]/I(V_{1})\cap\cdots\cap
I(V_{r})=\mathbb{C}[X_{1},\ldots,X_{r}]/I(V_{1})\times\cdots\times
\mathbb{C}[X_{1},\ldots,X_{r}]/I(V_{r})
\]
where $\mathbb{C}[X_{1},\ldots,X_{r}]/I(V_{j})$ is an integral domain since
$V_{j}$ is irreducible.

\bigskip

\noindent Laboratoire de Math\'{e}matiques et Physique Th\'{e}orique (UMR CNRS
6083)\newline Universit\'{e} Fran\c{c}ois-Rabelais, Tours \newline
F\'{e}d\'{e}ration de Recherche Denis Poisson - CNRS\newline Parc de
Grandmont, 37200 Tours, France.

\bigskip

\noindent Univ. Orl\'eans, MAPMO (UMR CNRS 7349), F-45067, Orl\'eans, France
\newline FDP - FR CNRS 2964

\end{document}